\documentclass{amsart}

\usepackage{amsmath}
\usepackage{amssymb}
\usepackage{amsthm}
\usepackage{esint}
\usepackage[inline]{enumitem}
\theoremstyle{plain}
\newtheorem{theorem}{Theorem}[section]
\newtheorem{corollary}[theorem]{Corollary}
\newtheorem{lemma}[theorem]{Lemma}
\newtheorem{proposition}[theorem]{Proposition}

\theoremstyle{definition}
\newtheorem{definition}[theorem]{Definition}
\newtheorem{remark}[theorem]{Remark}

\theoremstyle{remark}

\numberwithin{theorem}{section}
\numberwithin{equation}{section}

\newcommand{\N}{\mathbb{N}}

\newcommand{\Z}{\mathbb{Z}}

\newcommand{\R}{\mathbb{R}}

\newcommand{\dist}{\mathrm{dist}}
\newcommand{\diam}{\mathrm{diam}}

\newcommand{\cl}{\overline}
\newcommand{\loc}{\mathrm{loc}}

\DeclareMathOperator{\divergence}{div}

\newcommand{\laplacian}{\Delta}
\DeclareMathOperator*{\essinf}{ess\,inf}

\DeclareMathOperator*{\spt}{supp}
\newcommand{\capacity}{\mathrm{cap}}

\newcommand{\trinorm}[1]
{{
    \left\vert\kern-0.20ex\left\vert\kern-0.20ex\left\vert
    #1 
    \right\vert\kern-0.20ex\right\vert\kern-0.20ex\right\vert
}}

\newcommand{\M}{\mathcal{M}}

\newcommand{\A}{\mathcal{A}}
\newcommand{\W}{{\bf{W}}}
\newcommand{\w}{\mathcal{W}}
\begin{document}
%\begin{frontmatter}

%%%%%%%%%%%%%%%%%%%%%%%%%%%%%%%%%%%%%%%%
% Journal name, 
%%%%%%%%%%%%%%%%%%%%%%%%%%%%%%%%%%%%%%%%

%\journal{Advances in Mathematics}

%%%%%%%%%%%%%%%%%%%%%%%%%%%%%%%%%%%%%%%%
% Author's name, Title, Date and other information
%%%%%%%%%%%%%%%%%%%%%%%%%%%%%%%%%%%%%%%%

%% AMSART 
\title[Trace inequalities]{Trace inequalities of the Sobolev type and nonlinear Dirichlet problems}
\author{Takanobu Hara}
\email{takanobu.hara.math@gmail.com}
\address{Department of Mathematics, Hokkaido University, Kita 8 Nishi 10  Sapporo, Hokkaido 060-0810, Japan}
\date{\today}

%%% ELS
%\title{Trace inequalities of the Sobolev type and nonlinear Dirichlet problems \tnoteref{t1}}
%\author{Takanobu Hara\corref{cor1}\fnref{fn1}}
%\ead{takanobu.hara.math@gmail.com}
%\address{Department of Mathematics, Hokkaido University, Kita 8 Nishi 10  Sapporo, Hokkaido 060-0810, Japan}

%%%%%%%%%%%%%%%%%%%%%%%%%%%%%%%%%%%%%%%%
% Abstract
%%%%%%%%%%%%%%%%%%%%%%%%%%%%%%%%%%%%%%%%

\begin{abstract}
We discuss the solvability of nonlinear Dirichlet problems of the type
$- \laplacian_{p, w} u = \sigma$ in $\Omega$;
$u = 0$ on $\partial \Omega$,
where $\Omega$ is a bounded domain in $\R^{n}$,
$\laplacian_{p, w}$ is a weighted $(p, w)$-Laplacian and
$\sigma$ is a nonnegative locally finite Radon measure on $\Omega$.
We do not assume the finiteness of $\sigma(\Omega)$.
We revisit this problem from a potential theoretic perspective
and provide criteria for the existence of solutions
by $L^{p}(w)$-$L^{q}(\sigma)$ trace inequalities or capacitary conditions.
Additionally, we apply the method to the singular elliptic problem
$- \laplacian_{p, w} u = \sigma u^{- \gamma}$ in $\Omega$;
$u = 0$ on $\partial \Omega$
and derive connection with the trace inequalities.
%The results are new even for the Laplace operator.
\end{abstract}

\thanks{This work was supported by JSPS KAKENHI Grant Number JP18J00965 and JP17H01092.}
%\begin{keyword} %% keywords here, in the form: keyword \sep keyword
%Quasilinear elliptic equation \sep $p$-Laplacian \sep Measure data \sep Trace inequality \sep Singular elliptic problems
%MSC codes here, in the form: \MSC code \sep code or \MSC[2008] code \sep code (2000 is the default)
%\MSC[2020] 35J92 \sep 35J25 \sep 31C15 \sep 31C45.
%\end{keyword}
%\end{frontmatter}

%%%%%%%%%%%%%%%%%%%%%%%%%%%%%%%%%%%%%%%%
% Title & Table
%%%%%%%%%%%%%%%%%%%%%%%%%%%%%%%%%%%%%%%%

\maketitle
% \setcounter{tocdepth}{1}
% \tableofcontents

%%%%%%%%%%%%%%%%%%%%%%%%%%%%%%%%%%%%%%%%
% Body
%%%%%%%%%%%%%%%%%%%%%%%%%%%%%%%%%%%%%%%%

%%%%%%%%%%%%%%%%%%%%%%%%%%%%%%%%%%%%%%%%
\section{Introduction and main results}\label{sec:introduction}
%%%%%%%%%%%%%%%%%%%%%%%%%%%%%%%%%%%%%%%%

Let $\Omega$ be a bounded domain in $\R^{n}$, and let $1 < p < \infty$.
We consider the existence problem of positive solutions to quasilinear elliptic equations of the type
\begin{equation}\label{eqn:p-laplace}
\begin{cases}
\displaystyle
- \divergence \A(x, \nabla u) = \sigma
& \text{in} \ \Omega,
\\
u = 0
& \text{on} \ \partial \Omega,
\end{cases}
\end{equation}
where $- \divergence \A(x, \nabla \cdot)$ is a weighted $(p, w)$-Laplacian type elliptic operator,
$w$ is a $p$-admissible weight on $\R^{n}$ (see Sect. \ref{sec:preliminaries} for details)
and $\sigma$ is a nonnegative (locally finite) Radon measure on $\Omega$.
We do not assume the global finiteness of $\sigma$.
For the standard theory of quasilinear Dirichlet problems with finite signed measure data,
we refer to \cite{MR1025884, MR1205885, MR1354907, MR1386213, MR1409661, MR1760541, MR3676369}.
See also \cite{MR1205885, MR1264000, MR1655522, MR1955596, MR2859927}
for the definitions of local solutions ($\A$-superharmonic functions or locally renormalized solutions)
and their properties.

Most studies of the quasilinear measure data problem \eqref{eqn:p-laplace} assume the finiteness of $\sigma$
to ensure the existence of solutions satisfying the Dirichlet boundary condition.
However, the global finiteness of $\sigma$ is not a necessary condition,
and this existence problem has been stated as an open problem in a paper by Bidaut-V\'{e}ron (see, \cite[Problem 2]{MR1990293}).
If we recall classical potential theory (see, e.g., \cite{MR0350027, MR1801253}),
the solution $u$ to the Poisson equation $- \laplacian u = \sigma$ in $\Omega$; $u = 0$ on $\partial \Omega$ is given by the Green potential
\[
u(x) = \int_{\Omega} G_{\Omega}(x, y) \, d \sigma(y).
\]
Since the Green function $G_{\Omega}(\cdot, \cdot)$ vanishes on the boundary of $\Omega$,
the integral may be finite even if $\sigma$ is not finite.
In addition, the pointwise estimate of the Green function yields more concrete existence results
if the boundary of $\Omega$ is $C^{2}$;
$u \not \equiv + \infty$ if and only if $\int_{\Omega} \dist(x, \partial \Omega) \, d  \sigma(x) < \infty$ (see, e.g., \cite{MR3156649}).
However, finding the estimate is another problem,
and furthermore, the method is completely useless for nonlinear equations.
Note that the use of the Green function is one of the methods to solve the problem
and that there are rich examples of solutions even for nonlinear equations with infinite measure data.
%We can confirm this fact from observation of the Riesz measures of $p$-superharmonic functions.

One natural desire is to apply the theory of finite measure data problems to infinite measures,
but it is actually not enough, because the above examples do not necessarily satisfy the Dirichlet boundary condition in the traditional sense.
Therefore, we should adopt a two-step strategy;
\begin{enumerate*}[label=(\roman*)]
\item
find a local solution that satisfies the equation in a generalized sense, and 
\item
confirm the boundary condition.
\end{enumerate*}

Variational methods (more generally, the theory of monotone operators)
and comparison principles are effective for good measures.
Hence, we find the solution $u$ to \eqref{eqn:p-laplace} by 
\[
u(x)
=
\sup \left\{ v(x) \in H_{0}^{1, p}(\Omega; w) \cap \mathcal{SH}(\Omega) \colon 0 \le - \divergence \A(x, \nabla v) \le \sigma \right\},
\]
where $H_{0}^{1, p}(\Omega; w)$ is a weighted Sobolev space
and $\mathcal{SH}(\Omega)$ is the set of all $\A$-superharmonic functions in $\Omega$.
Perron's method for $\A$-superharmonic functions have been well studied
(see, e.g., \cite{MR863533, MR2305115}),
and the relation between them and their Riesz measures is also known (\cite{MR1205885, MR1890997}).
Furthermore, if $\sigma$ is absolutely continuous with respect to the $(p, w)$-capacity,
then the above set contains sufficiently many $\A$-superharmonic functions.
As a result, $u$ is a local solution if it is not identically infinite.
In \cite{MR3567503}, Cao and Verbitsky proved a comparison principle leading to the minimality of $u$
(see Theorem \ref{thm:comparison_principle} below for a refinement of it).
From this, $u$ can be considered to satisfy the boundary condition in a very weak sense.
Therefore, the remaining problem is to present a sufficient condition for $u$ to not be identically infinite.

In previous work \cite{MR4309811}, the author proved a decomposition theorem of measures
and provided an existence theorem for equations of the type \eqref{eqn:p-laplace}.
The sufficient condition was given by the $L^{p}(w)$-$L^{q}(\sigma)$ trace inequality of Sobolev type
\begin{equation}\label{eqn:trace_ineq}
\| f \|_{L^{q}(\Omega; \sigma)} \le C_{1} \| \nabla f \|_{L^{p}(\Omega; w)}, \quad \forall f \in C_{c}^{\infty}(\Omega),
\end{equation}
where
\[
0 < q < p.
\]
In general, the measure $\sigma$ satisfying \eqref{eqn:trace_ineq} is not necessarily finite.
Note that if \eqref{eqn:trace_ineq} holds, then $\sigma$ must be absolutely continuous with respect to the $(p, w)$-capacity.
Our first existence theorem is as follows.

\begin{theorem}\label{thm:main_theorem}
Assume that $\A$ satisfies \eqref{eqn:coercive}-\eqref{eqn:homogenity}.
Let $\sigma$ be a nonnegative Radon measure on $\Omega$, and let $0 < q < p$.
Assume that \eqref{eqn:trace_ineq} holds, and let $C_{1}$ be the best constant.
Then, there exists a minimal nonnegative $\mathcal{A}$-superharmonic solution $u$ to \eqref{eqn:p-laplace}
satisfying
\begin{equation}\label{eqn:energy_bound}
\frac{ (\alpha / \beta)^{p} }{\alpha} C_{1}^{q}
\le
\| \nabla u^{\frac{p - 1}{p - q}} \|_{L^{p}(\Omega; w)}^{p - q}
\le
\frac{1}{q} \left( \frac{p - 1}{p - q} \right)^{p - 1} \frac{1}{\alpha} C_{1}^{q}
\end{equation}
and
\begin{equation}\label{eqn:COV-bound}
\frac{ (\alpha / \beta) }{ \alpha^{ \frac{1}{p} } }
C_{1}
\le
\left( \int_{\Omega} u^{ \frac{q(p - 1)}{p - q} }  \, d \sigma \right)^{\frac{p - q}{qp}}
\le
\frac{1}{q^{ \frac{1}{p} }}
\left( \frac{p - 1}{p - q} \right)^{\frac{p - 1}{p}}
\frac{1}{ \alpha^{ \frac{1}{p} } }
C_{1}.
\end{equation}
In particular, $u$ satisfies the Dirichlet boundary condition in the sense that
\begin{equation}\label{eqn:BC@main_theorem}
u^{\frac{p - 1}{p - q}} \in H_{0}^{1, p}(\Omega; w).
\end{equation}
Conversely, if there exists a nonnegative $\mathcal{A}$-superharmonic solution $u$ to \eqref{eqn:p-laplace}
satisfying \eqref{eqn:BC@main_theorem}, then \eqref{eqn:trace_ineq} holds.
\end{theorem}

For $q = 1$, Theorem \ref{thm:main_theorem} is well-known,
because \eqref{eqn:trace_ineq} is necessary and sufficient for $\sigma \in (H_{0}^{1, p}(\Omega; w))^{*}$.
Claims that are equivalent to these were first given by Seesanea and Verbitsky \cite{MR3881877} for linear uniformly elliptic operators.
In \cite{MR4309811}, their arguments were applied to weighted $(p, w)$-Laplace operators,
and Theorem \ref{thm:main_theorem} was proved with a different formulation.
We give a direct proof of the existence theorem for \eqref{eqn:p-laplace}
and further discuss what occurs when condition \eqref{eqn:trace_ineq} is relaxed.
%In particular, we observe that \eqref{eqn:p-laplace} has a solution for a class of measures that is truly larger than a finite measure.
In addition, we present a more concrete existence result using Hardy-type inequalities (see Corollary \ref{cor:sufficient_condition}).

Condition \eqref{eqn:BC@main_theorem} has already appeared in the study of elliptic equations with singular nonlinearity
since the work by Boccardo and Orsina \cite{MR2592976}.
See, e.g., \cite{MR2718666, MR3450747, MR3462443, MR3478284, MR3712944} and the references therein.
The lower-order terms in these studies have a structure that increases near the boundary,
which affects the boundary behavior of solutions via the trace inequalities.
In Sect. \ref{sec:app}, we apply our framework to the singular elliptic problem
and examine the implication of condition \eqref{eqn:BC@main_theorem} (see Corollary \ref{cor:BO}).
Also, we prove the following criterion.  %characterization for \eqref{eqn:trace_ineq}.

\begin{theorem}\label{thm:singular_fe}
%Let $\Omega$ be a bounded domain in $\R^{n}$.
Let $1 < p < \infty$, and let $0 < q < 1$.
Suppose that $\sigma \in \M^{+}_{0}(\Omega) \setminus \{ 0 \}$.
Then there exists a unique finite energy weak solution $u \in H_{0}^{1, p}(\Omega; w)$ to
\begin{equation}\label{eqn:semilin}
\begin{cases}
- \divergence \A(x, \nabla u) = \sigma u^{q - 1}
& \text{in} \ \Omega,
\\
u > 0
& \text{in} \ \Omega,
\\
u = 0  & \text{on} \ \partial \Omega
\end{cases}
\end{equation}
if and only if \eqref{eqn:trace_ineq} holds.
\end{theorem}

The sufficiency part of this theorem is a generalization of \cite[Theorem 5.1]{MR2592976}.
The necessity part seems to be new even if $\divergence \A(x, \nabla u) = \laplacian u$.
Oliva and Petitta \cite{MR3712944} mentioned a related characterization, but the connection with \eqref{eqn:trace_ineq} was not considered.
A restricted similar result was proved by Bal and Garain \cite{bal2021weighted}.
The case of $1 \le q < p$ has been treated in \cite[Theorems 1.2 and 1.3]{MR4309811}.

Additionally, we prove  a nonlinear version of \cite[Theorem 6]{MR1866062}.

\begin{theorem}\label{thm:bdd_sol}
Let $\sigma \in \M^{+}_{0}(\Omega) \setminus \{ 0 \}$.
Let $h \colon (0, \infty) \to (0, \infty)$ be a continuously differentiable nonincreasing function.
Assume that there exists a continuous weak supersolution $v \in H^{1, p}_{\loc}(\Omega; w) \cap C( \cl{\Omega} )$ to \eqref{eqn:p-laplace}.
Then there exists a unique continuous weak solution $u \in H^{1, p}_{\loc}(\Omega; w) \cap C( \cl{\Omega} )$ to
\begin{equation}\label{eqn:singular}
\begin{cases}
\displaystyle
- \divergence \A(x, \nabla u) = \sigma h(u) 
& \text{in} \ \Omega,
\\
u > 0
& \text{in} \ \Omega,
\\
u = 0
&
\text{on} \ \partial\Omega.
\end{cases}
\end{equation}
Conversely, if there exists a continuous weak supersolution $u \in H^{1, p}_{\loc}(\Omega; w) \cap C( \cl{\Omega} )$ to \eqref{eqn:singular},
then there exists a continuous weak solution $v \in H^{1, p}_{\loc}(\Omega; w) \cap C( \cl{\Omega} )$ to \eqref{eqn:p-laplace}. 
\end{theorem}

General existence results of classical solutions to elliptic partial differential equations with singular nonlinearity
were first established by Crandall, Rabinowitz and Tartar \cite{MR427826}.
In  \cite{MR1866062}, M\^{a}agli and Zribi proved an existence theorem for $\divergence \A(x, \nabla u) = \laplacian u$ and $\sigma$ in the Kato class
and presented the bounds \eqref{eqn:bound_sing} and \eqref{eqn:bound_for_sing_eq} using the Green potential of $\sigma$.
We replace the Green potential with the solution to \eqref{eqn:p-laplace}.
%For existence results of classical solutions to Poisson's equation, we refer to \cite{MR856511}.
See also one-dimensional results by Taliaferro \cite{MR548961} and further developments \cite{MR1020724, MR1400567}.
%Note that the measure $\sigma$ may not be finite generally. 

%For simplicity, we assume that $\Omega$ is bounded.
%However, our argument can also be applied for $\Omega = \R^{n}$ if $1 < p < n$
%by replacing the weighted Sobolev space with the unweighted homogeneous one.
%In that case, the increase in $\sigma$ near the boundary is replaced by decay at infinity.
%For singular elliptic problems with the Laplacian, see \cite{MR1458503}.
%For asymptotic behavior of Newtonian or Wolff potentials, see also \cite{MR1141779, MR3567503}.

\subsection*{Organization of the paper}
In Sect. \ref{sec:preliminaries}, we present auxiliary results from nonlinear potential theory.
In Sect. \ref{sec:potentials} we provide a framework to solve \eqref{eqn:p-laplace}.
In Sect. \ref{sec:trace}, we prove Theorem \ref{thm:main_theorem}.
In Sect. \ref{sec:weak-trace}, we extend Theorem \ref{thm:main_theorem} using capacitary conditions.
This section is independent of Sect. \ref{sec:app}.
In Sect. \ref{sec:app}, we apply the framework in Sect. \ref{sec:potentials} to Eq. \eqref{eqn:singular}
and prove Theorems \ref{thm:singular_fe} and \ref{thm:bdd_sol} as a consequence.

\subsection*{Notation}
We use the following notation.
Let $\Omega$ be a domain (connected open subset) in $\R^{n}$. 
\begin{itemize}
\item
$\mathbf{1}_{E}(x) :=$ the indicator function of a set $E$.
\item
$C_{c}^{\infty}(\Omega) :=$
the set of all infinitely-differentiable functions with compact support in $\Omega$.
\item
$\M^{+}(\Omega) :=$ the set of all nonnegative Radon measures on $\Omega$.
\item
$L^{p}(\Omega; \mu) :=$ the $L^{p}$ space with respect to $\mu \in \M^{+}(\Omega)$.
\end{itemize}
For simplicity, we often write $L^{p}(\Omega; \mu)$ as $L^{p}(\mu)$.
For a ball $B = B(x, R)$ and $\lambda > 0$, $\lambda B := B(x, \lambda R)$.
For measures $\mu$ and $\nu$, we denote $\nu \le \mu$ if $\mu - \nu$ is a nonnegative measure.
For a sequence of extended real valued functions $\{ f_{j} \}_{j = 1}^{\infty}$,
we denote $f_{j} \uparrow f$
if $f_{j + 1} \ge f_{j}$ for all $j \ge 1$ and $\lim_{j \to \infty} f_{j} = f$.
The letters $c$ and $C$ denote various constants with and without indices.

%%%%%%%%%%%%%%%%%%%%%%%%%%%%%%%%%%%%%%%%
\section{Preliminaries}\label{sec:preliminaries}
%%%%%%%%%%%%%%%%%%%%%%%%%%%%%%%%%%%%%%%%

\subsection{Weighted Sobolev spaces}

First, we recall basics of nonlinear potential theory from \cite{MR2305115}.
Let $1 < p < \infty$ be a fixed constant.
A Lebesgue measurable function $w$ on $\R^{n}$
is said to be the \textit{weight} on $\R^{n}$ if $w \in L^{1}_{\loc}(\R^{n}; dx)$ and $w(x) > 0$ $dx$-a.e.
We write $w(E) = \int_{E} w \, dx$ for a Lebesgue measurable set $E \subset \R^{n}$.
We always assume that $w$ is \textit{$p$-admissible},
that is, positive constants $C_{D}$, $C_{P}$ and $\lambda \ge 1$ exist, such that
\[
w(2B) \le C_{D} w(B)
\]
and
\[
\fint_{B} |f - f_{B}| \, dw \le C_{P} \, \diam(B) \left( \fint_{\lambda B} |\nabla f|^{p} \, dw \right)^{\frac{1}{p}},
\quad
\forall f \in C_{c}^{\infty}(\R^{n}),
\]
where $B$ is an arbitrary ball in $\R^{n}$,
$\fint_{B} = w(B)^{-1} \int_{B}$
and
$f_{B} = \fint_{B} f \, dw$.
For the basic properties of $p$-admissible weights,
see \cite[Chapter A.2]{MR2867756}, \cite[Chapter 20]{MR2305115} and the references therein.
Every Muckenhoupt $A_{p}$-weight is $p$-admissible.
%If $n = 1$, then $p$-admissible weights are $A_{p}$-weights conversely.
One important property of $p$-admissible weights is the Sobolev inequality. %(\cite{MR1098839, MR1150597}).
In particular, the following form of the Poincar\'{e} inequality holds:
\begin{equation*}\label{eqn:poincare}
\int_{B} |f|^{p} \, dw
\le
C \, \diam(B)^{p} \int_{B} |\nabla f|^{p} \, dw,
\quad
\forall f \in C_{c}^{\infty}(B),
\end{equation*}
where $C$ is a constant depending only on $p$, $C_{D}$, $C_{P}$ and $\lambda$.

Let $\Omega$ be a bounded domain in $\R^{n}$.
The weighted Sobolev space $H^{1, p}(\Omega; w)$ is the closure of $C^{\infty}(\Omega)$
with respect to the norm
\[
\| u \|_{H^{1, p}(\Omega; w)}
:=
\left(
\int_{\Omega} |u|^{p} + |\nabla u|^{p} \, d w
\right)^{\frac{1}{p}}.
\]
The corresponding local space $H^{1, p}_{\loc}(\Omega; w)$ is defined in the usual manner.
We denote by $H_{0}^{1, p}(\Omega; w)$ the closure of $C_{c}^{\infty}(\Omega)$ in $H^{1, p}(\Omega; w)$.
Since $\Omega$ is bounded, we can take $\| \nabla \cdot \|_{L^{p}(\Omega; w)}$ 
as the norm of $H_{0}^{1, p}(\Omega; w)$ by the Poincar\'{e} inequality.

Let $\Omega \subset \R^{n}$ be open and let $K \subset \Omega$ be compact.
The \textit{(variational) $(p, w)$-capacity} $\capacity_{p, w}(K, \Omega)$
of the condenser $(K, \Omega)$ is defined by
\[
\capacity_{p, w}(K, \Omega)
:=
\inf \left\{
\| \nabla u \|_{L^{p}(\Omega; w)}^{p} \colon u \geq 1 \ \text{on} \ K, \ u \in C_{c}^{\infty}(\Omega)
\right\}.
\]
Since $\Omega \subset \R^{n}$ is bounded,
$\capacity_{p, w}(E, \Omega) = 0$ if and only if $C_{p, w}(E) = 0$,
where $C_{p, w}(\cdot)$ is the (Sobolev) capacity of $E$.
% (see \cite[Lemma 6.15]{MR2867756}).
We say that a property holds \textit{quasieverywhere} (q.e.)
if it holds except on a set of $(p, w)$-capacity zero.
An extended real valued function $u$ on $\Omega$ is called as \textit{quasicontinuous}
if for every $\epsilon > 0$ there exists an open set $G$ such that
$C_{p, w}(G) < \epsilon$ and $u|_{\Omega \setminus G}$ is continuous.
Every $u \in H^{1, p}_{\loc}(\Omega; w)$ has a quasicontinuous representative $\tilde{u}$
such that $u = \tilde{u}$ a.e.

We denote by $\M^{+}_{0}(\Omega)$ the set of all Radon measures $\mu$
that are absolutely continuous with respect to the $(p, w)$-capacity.
If $\mu \in \M^{+}_{0}(\Omega)$ is finite,
then the integral $\int_{\Omega} f \, d \mu$ is well-defined for any $(p, w)$-quasicontinuous function $f$ on $\Omega$.

\subsection{$\A$-superharmonic functions}

For $u \in H^{1, p}_{\loc}(\Omega; w)$,
we define the $\A$-Laplace operator $\divergence \A(x, \nabla \cdot)$ by
\[
\langle - \divergence \A(x, \nabla u), \varphi \rangle
=
\int_{\Omega} \A(x, \nabla u) \cdot \nabla \varphi \, dx,
\quad
\forall \varphi \in C_{c}^{\infty}(\Omega).
\]
The precise assumptions on $\A \colon \Omega \times \R^{n} \to \R^{n}$
are as follows:
For each $z \in \R^{n}$, $\A(\cdot, z)$ is measurable, for each $x \in \Omega$, $\A(x, \cdot)$ is continuous,
and there exist $0 < \alpha \le \beta < \infty$ such that
\begin{align}
\A(x, z) \cdot z \ge \alpha w(x) |z|^{p}, \label{eqn:coercive}
\\ 
|\A(x, z)| \le \beta w(x) |z|^{p - 1},  \label{eqn:growth}
\\
\left( \A(x, z_{1}) - \A(x, z_{2}) \right) \cdot (z_{1} - z_{2}) > 0, \label{eqn:monotonicity}
\\ 
\A(x, t z) = t |t|^{p - 2} \A(x, z) \label{eqn:homogenity}
\end{align}
for all $x \in \Omega$, $z, z_{1}, z_{2} \in \R^{n}$, $z_{1} \not = z_{2}$ and $t \in \R$.
If $u \in H^{1, p}_{\loc}(\Omega; w)$ satisfies
\[
\int_{\Omega} \A(x, \nabla u) \cdot \nabla \varphi \, dx = (\ge) \, 0,
\quad
\forall \varphi \in C_{c}^{\infty}(\Omega), \ \varphi \ge 0,
\]
then it is called a \textit{weak solution (supersolution)} to
$- \divergence \A(x, \nabla u) = 0$ in $\Omega$.

A function $u \colon \Omega \to ( - \infty, \infty]$ is called \textit{$\A$-superharmonic} if
$u$ is lower semicontinuous in $\Omega$, is not identically infinite,
and satisfies the comparison principle on each subdomain $D \Subset \Omega$;
if $h \in H^{1, p}_{\loc}(D; w) \cap C(\cl{D})$ is a continuous weak solution to $- \divergence \A(x, \nabla u) = 0$ in $D$
and if $u \geq h$ on $\partial D$, then $u \geq h$ in $D$.

If $u$ is an $\A$-superharmonic function in $\Omega$,
then for any $k > 0$, $\min\{ u, k \}$ is a weak supersolution to $- \divergence \A(x, \nabla u) = 0$ in $\Omega$.
Conversely, if $u$ is a weak supersolution to $- \divergence \A(x, \nabla u) = 0$ in $\Omega$,
then its \textit{lsc-regularization}
\[
u^{*}(x)
:=
\lim_{r \to 0} \essinf_{B(x, r)} u
\]
is $\A$-superharmonic in $\Omega$.
If $u$ and $v$ are $\A$-superharmonic in $\Omega$ and $u(x) \leq v(x)$ for a.e. $x \in \Omega$,
then $u \leq v$ in the pointwise sense.
%Every $(p, w)$-superharmonic function known to be quasicontinuous.
%In particular, the set $\{ u = \infty \}$ has zero $(p, w)$-capacity whenever $u$ is $(p, w)$-superharmonic.
%Let $u$ be an $\A$-superharmonic function in $\Omega$.
%Then for all $k > 0$, $u_{k} := \min\{ u, k \} \in H^{1, p}_{\loc}(\Omega; w)$
%and $- \divergence \A(x, \nabla u_{k}) \ge 0$ in $\Omega$ in the sense of supersolutions.

A Radon measure $\mu = \mu[u]$ is called the \textit{Riesz measure} of $u$ if
\[
\lim_{k \to \infty}
\int_{\Omega} \A(x, \nabla \min\{ u, k\}) \cdot \nabla \varphi \, dx
=
\int_{\Omega} \varphi \, d \mu,
\quad
\forall \varphi \in C_{c}^{\infty}(\Omega).
\]
It is known that every $\A$-superharmonic function has a unique Riesz measure.

The following weak continuity result was given by Trudinger and Wang \cite{MR1890997}.

\begin{theorem}[{\cite[Theorem 3.1]{MR1890997}}]\label{thm:TW}
Suppose that $\{ u_{k} \}_{k = 1}^{\infty}$ is a sequence of nonnegative $\A$-superharmonic functions in $\Omega$.
Assume that $u_{k} \to u$ a.e. in $\Omega$ and that $u$ is $\A$-superharmonic in $\Omega$.
Let $\mu[u_{k}]$ and $\mu[u]$ be the Riesz measures of $u_{k}$ and $u$, respectively.
Then $\mu[u_{k}]$ converges to $\mu[u]$ weakly, that is,
\[
\int_{\Omega} \varphi \, d \mu[u_{k}] \to \int_{\Omega} \varphi \, d \mu[u],
\quad
\forall \varphi \in C_{c}^{\infty}(\Omega).
\]
\end{theorem}

The Harnack-type convergence theorem follows from combining Theorem \ref{thm:TW} and \cite[Lemma 7.3]{MR2305115}:
If $\{ u_{k} \}_{k = 1}^{\infty}$ is a nondecreasing sequence of $\A$-superharmonic functions in $\Omega$
and if $u := \lim_{k \to \infty} u_{k} \not \equiv \infty$, 
then $u$ is $\A$-superharmonic in $\Omega$ and $\mu[u_{k}]$ converges to $\mu[u]$ weakly.
%This fact will be used repeatedly via Theorem \ref{thm:potentials} below.

%The two-sided Wolff potential estimate for $\A$-superharmonic functions was established by
%Kilpel\"{a}inen and Mal\'{y} \cite{MR1205885,MR1264000}.
%The following weighted version is from Mikkonenn \cite{MR1386213}.
%%See also \cite{MR1890997,MR3842214} for other proofs.
%
%\begin{theorem}[{\cite[Theorem 3.1]{MR1386213}}]\label{thm:wolff}
%Assume that $u \ge 0$ is $\A$-superharmonic in $B(x, 2R)$.
%Let $\mu$ be the Riesz measure of $u$.
%Then,
%\[
%\frac{1}{C} \W_{1, p, w}^{R} \mu(x)
%\leq
%u(x)
%\leq
%C \left(
%\inf_{B(x, R)} u
%+
%\W_{1, p, w}^{2R} \mu(x)
%\right),
%\] 
%where $C \ge 1$ is a constant depending only on $p$, $C_{D}$, $C_{P}$ and $\lambda$, and
%$\W_{1, p, w}^{R} \mu$ is the \textit{truncated Wolff potential} of $\mu$, which is defined by
%\begin{equation}\label{eqn:def_wolff}
%\W_{1, p, w}^{R} \mu (x)
%:=
%\int_{0}^{R}
%\left(
%r^{p} \frac{ \mu(B(x, r))}{ w(B(x, r)) }
%\right)^{\frac{1}{p - 1}}
%\frac{dr}{r}.
%\end{equation}
%\end{theorem}

\section{Minimal $\A$-superharmonic solution to \eqref{eqn:p-laplace}}\label{sec:potentials}

Next, we introduce classes of smooth measures. For detail, see \cite{MR4309811} and the references therein.
%Following the theory of Dirichlet forms (see, e.g., \cite[Chapter 2]{MR2778606}), we introduce classes of smooth measures.
%Let $S_{0}$ be the set of all Radon measures $\mu$ satisfying
%\[
%\int_{\Omega} \varphi \, d \mu
%\le
%C \| \nabla \varphi \|_{L^{p}(w)},
%\quad \forall \varphi \in C_{c}^{\infty}(\Omega), \ \varphi \ge 0,
%\]
%where $C$ is a positive constant.

\begin{definition}
For $\mu \in (H_{0}^{1, p}(\Omega))^{*} \cap \M^{+}(\Omega)$,
we denote by $\w_{\A}^{0} \mu$ the lsc-regularization of  the weak solution $u \in H_{0}^{1, p}(\Omega; w)$ to
\[
\int_{\Omega} \A (x, \nabla u) \cdot \nabla \varphi \, dx = \langle \mu, \varphi \rangle,
\quad
\forall \varphi \in H_{0}^{1, p}(\Omega; w).
\]
\end{definition}

Furthermore, we define a class of smooth measures $S_{c}[\mathcal{A}](\Omega)$ by
\[
S_{c}[\mathcal{A}](\Omega)
:=
\left\{
\mu \in (H_{0}^{1, p}(\Omega))^{*} \cap \M^{+}(\Omega) \colon \sup_{\Omega} \w^{0}_{\A} \mu < \infty \ \text{and} \ \spt \mu \Subset \Omega
\right\}.
\]
%The two-sided Wolff potential estimate for $\A$-superharmonic functions was established by
By the two-sided Wolff potential estimate for $\A$-superharmonic functions due to Kilpel{\"a}inen and Mal{\'y} (see \cite{MR1264000,MR1386213}),
if $u \ge 0$ is $\A$-superharmonic in $B(x, 2R)$ and if $\mu$ is the Riesz measure of $u$,
then,
\[
\frac{1}{C} \W_{1, p, w}^{R} \mu(x)
\leq
u(x)
\leq
C \left(
\inf_{B(x, R)} u
+
\W_{1, p, w}^{2R} \mu(x)
\right),
\] 
where $C = C(p, C_{D}, C_{P}, \lambda)$ and
$\W_{1, p, w}^{R} \mu$ is the \textit{truncated Wolff potential} of $\mu$, which is defined by
\begin{equation}\label{eqn:def_wolff}
\W_{1, p, w}^{R} \mu (x)
:=
\int_{0}^{R}
\left(
r^{p} \frac{ \mu(B(x, r))}{ w(B(x, r)) }
\right)^{\frac{1}{p - 1}}
\frac{dr}{r}.
\end{equation}
Using this estimate twice, we can prove that 
$S_{c}[\mathcal{A}](\Omega)=  S_{c}[w(x) |\nabla \cdot|^{p - 2} \nabla \cdot](\Omega)$.
Below, we write $S_{c}[\mathcal{A}](\Omega)$ as $S_{c}(\Omega)$ for simplicity.

%\begin{proposition}
%$S_{c}(\Omega)[\A](\Omega) =  \s_c[w(x) |\nabla \cdot|^{p - 2} \nabla \cdot](\Omega)$.
%\end{proposition}
%
%\begin{proof}
%Suppose that $\mu \in \s_c[w(x) |\nabla \cdot|^{p - 2} \nabla \cdot](\Omega)$.
%Let
%\[
%u = \w^{0}_{\A} \mu \quad \text{and} \quad v = \w^{0}_{w(x) |\nabla \cdot|^{p - 2} \nabla} \mu.
%\]
%Set $k = \| u \|_{L^{\infty}(\mu)}$. Then,
%\[
%\alpha \int_{\Omega} |\nabla (u - k)_{+}|^{p} \, dw
%\le
%\int_{\Omega} \A(x, \nabla u) \cdot \nabla (u - k)_{+} \, dx
%=
%\int_{\Omega} (u - k)_{+} \, d \mu
%= 0.
%\]
%Since $u$ is $\A$-superharmonic in $\Omega$,
%\[
%\sup_{\Omega} u \le k = \| u \|_{L^{\infty}(\mu)} = \sup_{ \spt \mu } u.
%\]
%Let $R = \dist( \spt \mu, \partial \Omega) / 4$.
%Fix $x \in \spt \mu$. Then,
%\[
%\begin{split}
%\alpha (\inf_{B(x, R)} u)^{p} \capacity_{p, w}(B(x, R), \Omega)
%& \le
%\int_{\Omega} \A(x, \nabla u) \cdot \nabla \min\{ u, \inf_{B(x, R)} u \} \, dx
%\\
%& \le
%(\inf_{B(x, R)} u) \, \mu(\Omega).
%\end{split}
%\]
%Therefore, according to Theorem \ref{thm:wolff} and the Poincar\'{e} inequality,
%\[
%\begin{split}
%u(x)
%& \le
%C \left( \inf_{B(x, R)} u + \W_{1, p, w}^{2R} \mu(x) \right)
%\\
%& \le
%C \left( \diam(\Omega)^{p} \frac{ \mu(\Omega) }{ w(B(x, R)) } \right)^{\frac{1}{p - 1}}
%+
%C v(x).
%\end{split}
%\]
%The right-hand side is bounded on $\spt \mu$; thus $\mu \in S_{c}(\Omega)[\A](\Omega)$.
%The converse inclusion follows in the same manner.
%\end{proof}

\begin{theorem}[{\cite{MR4309811}}]\label{thm:approximation}
Let $\mu \in \M^{+}(\Omega)$.
Then $\mu \in \M^{+}_{0}(\Omega)$ if and only if
there exists an increasing sequence of compact sets $\{ F_{k } \}_{k = 1}^{\infty}$ such that
$\mu_{k} := \mathbf{1}_{F_{k}} \mu \in S_{c}(\Omega)$ for all $k \ge 1$
and
$\mu\left( \Omega \setminus \bigcup_{k = 1}^{\infty} F_{k} \right) = 0$.
\end{theorem}

%\begin{example}
%If $\sigma$ is a nonnegative lower semicontinuous locally integrable function on $\Omega$,
%then $F_{k} = \{ f \le k \} \cap \cl{ \Omega_{k} }$ ($k \ge 1$) satisfy the above property, where $\{ \Omega_{k} \}_{k = 1}^{\infty}$ is an exhaustion of $\Omega$.
%\end{example}

\begin{definition}[\cite{MR3567503}]
Let $\mu$ be a nonnegative Radon measure on $\Omega$.
We say that a function $u$ is an \textit{$\A$-superharmonic solution (supersolution)}
to $- \divergence \A(x, \nabla u) = \mu$ in $\Omega$,
if $u$ is $\A$-superharmonic in $\Omega$ and $\mu[u] = \mu$ ($\mu[u] \ge \mu$),
where $\mu[u]$ is the Riesz measure of $u$.
We say that a nonnegative solution $u$ is \textit{minimal}
if $v \geq u$ in $\Omega$ whenever
$v$ is a nonnegative supersolution to the same equation.
\end{definition}

\begin{definition}
For $\mu \in \M^{+}_{0}(\Omega)$, we define %an extended real valued function $\w_{\A} \mu$ by
\[
\w_{\A} \mu(x) := \sup \left\{ \w_{\A}^{0} \nu(x) \colon \nu \in S_{c}(\Omega) \ \text{and} \  \nu \le \mu \right\}.
\]
\end{definition}

If $\mu \in (H_{0}^{1, p}(\Omega))^{*} \cap \M^{+}(\Omega)$, then $\w_{\A} \mu =  \w_{\A}^{0} \mu$.
If $\w_{\A} \mu \not \equiv \infty$,
then $u = \w_{\A} \mu$ is the minimal nonnegative $\A$-superharmonic solution to
$- \divergence \A(x, \nabla u) = \mu$ in $\Omega$.
From the argument in \cite[Theorem 3.1]{MR4309811}, the following comparison principle holds.

\begin{theorem}\label{thm:comparison_principle}
Let $\Omega$ be a bounded domain.
Let $v$ be a nonnegative $\A$-superharmonic function in $\Omega$ with the Riesz measure $\nu$.
Assume that $\mu \in \M^{+}_{0}(\Omega)$ and that $\mu \le \nu$.
Then $\w_{\A} \mu \leq v$ in $\Omega$.
\end{theorem}

%\begin{proof}
%We may assume that $\mu \in S_{c}(\Omega)$ without loss of generality.
%Let $\{ \Omega_{k} \}_{k = 1}^{\infty}$ be an exhaustion of $\Omega$.
%Set $v_{k} = \min\{ v, k \}$. Then
%\[
%\begin{cases}
%- \divergence \A(x, \nabla v_{k}) = \mathbf{1}_{ \{ v < k \} } \nu + \nu_{s}(k) & \text{in} \ \Omega_{k},
%\\
%0 \le v_{k} \in H^{1, p}(\Omega_{k}, w),
%\end{cases}
%\]
%where $\nu_{s}(k)$ is a nonnegative measure concentrated on $\{ u = k \}$.
%Let $u_{k}$ be the weak solution to
%\[
%\begin{cases}
%- \divergence \A(x, \nabla u_{k}) =  \mathbf{1}_{ \{ v < k \} } \mu & \text{in} \ \Omega_{k},
%\\
%u_{k} \in H_{0}^{1, p}(\Omega_{k}, w).
%\end{cases}
%\]
%We denote by $u_{k}$ the zero extension of the lsc-regularization of $u_{k}$ again.
%From the comparison principle for weak solutions, $0 \le u_{k} \le v_{k} \le v$ in $\Omega_{k}$.
%By the same reason, $u_{k} \le u_{k + 1}$ in $\Omega_{k}$.
%Set $\tilde{u}(x) = \lim_{k \to \infty} u_{k}(x)$.
%Since $\mu \in S_{c}(\Omega)$, $\{ u_{k} \}_{k = 1}^{\infty}$ is bounded in $H_{0}^{1, p}(\Omega; w)$.
%Thus, $\tilde{u} \in H_{0}^{1, p}(\Omega; w)$.
%On the other hand, since the set $\{ v = \infty \}$ has $(p, w)$-capacity, $\mathbf{1}_{ \{ v < k \} } \uparrow 1$ $\mu$-a.e. in $\Omega$.
%By Theorem \ref{thm:TW} and the monotone convergence theorem, the Riesz measure of $\tilde{u}$ is $\mu$.
%By the uniqueness of weak solutions, $\tilde{u} = u$.
%\end{proof}

By the minimality, we can regard $u$ as a solution to \eqref{eqn:p-laplace}.
We will discuss sufficient conditions for $\w_{\A} \mu \not \equiv \infty$ below.

%%%%%%%%%%%%%%%%%%%%%%%%%%%%%%%%%%%%%%%%
\section{Strong-type inequality}\label{sec:trace}
%%%%%%%%%%%%%%%%%%%%%%%%%%%%%%%%%%%%%%%%

Let us recall known results of the trace inequalities for the upper triangle case $q < p$.
For  the case $q \ge p$, we refer to \cite{MR1411441, MR2777530} and the references therein. % MR1643072
The first characterization of \eqref{eqn:trace_ineq} was given
by Maz'ya and Netrusov \cite{MR1313906} using a capacitary condition.
Cascante, Ortega and Verbitsky \cite{MR1734322} and Verbitsky \cite{MR1747901}
studied non-capacitary characterizations for inequalities of the type \eqref{eqn:trace_ineq}
based on the Hedberg-Wolff theorem.
For example, they proved the following:
If $\Omega = \R^{n}$ and $w = 1$, then the best constant $C_{1}$ in \eqref{eqn:trace_ineq} satisfies
\begin{equation}\label{eqn:COV-energy}
\frac{1}{c} \, C_{1}
\le
\left(
\int_{\R^{n}}
\left( \W_{1, p} \sigma \right)^{\frac{q(p - 1)}{p - q}}
\, d \sigma
\right)^{\frac{p - q}{qp}}
\le
c \, C_{1},
\end{equation}
where $c = c(n, p, q)$ and $\W_{1, p} \sigma = \W_{1, p, 1}^{\infty} \sigma$ is the Wolff potential of $\sigma$ (see \eqref{eqn:def_wolff}).

Verbitsky and his colleagues recently studied sublinear type elliptic problems of the form \eqref{eqn:COV-energy} for the case $1 < q < p$.
See \cite{MR3311903, MR3567503, MR4105916, MR3881877} and the references therein.
Their work shows a connection between the existence of positive solutions to \eqref{eqn:semilin},
energy conditions of the type of \eqref{eqn:COV-energy} (or equivalent \eqref{eqn:trace_ineq}) and certain weighted norm inequalities.
%They treated (semi-)linear equations on locally compact Hausdorff spaces and quasilinear equations on $\R^{n}$.
An extension of their result to $(p, w)$-Laplace equations on bounded domains was given by
the author \cite{MR4309811} (see also \cite{MR4079054}).
The counterpart of the Cascante-Ortega-Verbitsky theorem was also presented.
Theorem \ref{thm:main_theorem} is a variant of it.

Following the strategy in \cite{MR1747901}, we prove a logarithmic Caccioppoli type estimate.
For $(p, w)$-Laplace operators, it follows from the Picone-type inequality (see \cite{MR1618334,MR3273896}).

\begin{lemma}\label{lem:picone}
Let $\sigma \in S_{c}(\Omega)$, and let $u = \w_{\A} \sigma$.
Then
\[
\int_{\Omega} |f|^{p} \frac{d \sigma}{u^{p - 1}}
\le
\beta^{p} \alpha^{1 - p} \int_{\Omega} |\nabla f|^{p} \, d w,
\quad \forall f \in C_{c}^{\infty}(\Omega).
\]
\end{lemma}

\begin{proof}
Without loss of generality we may assume that $f \ge 0$.
Let $v = u + \epsilon$,  where $\epsilon$ is a positive constant.
Since $v \ge \epsilon > 0$, we have $f^{p} v^{1 - p} \in H_{0}^{1, p}(\Omega; w)$; hence,
\[
\int_{\Omega} f^{p} \frac{d \sigma}{v^{p - 1}}
=
\int_{\Omega} \A(x, \nabla v) \cdot \nabla (f^{p} v^{1 - p}) \, dx.
\]
By \eqref{eqn:coercive} and \eqref{eqn:growth}, for a.e. $x \in \Omega$,
\[
\begin{split}
\A(x, \nabla v(x)) \cdot \nabla (f(x)^{p} v(x)^{1 - p})
& =
p \A(x, \nabla v(x)) \cdot \nabla f(x) f(x)^{p - 1} v(x)^{1 - p}
\\
& \quad +
(1 - p) \A(x, \nabla v(x)) \cdot \nabla v(x) f(x)^{p} v(x)^{- p}
\\
& \le
p \beta w(x) |\nabla v(x)|^{p - 1} |\nabla f(x)| f(x)^{p - 1} v(x)^{1 - p}
\\
& \quad +
(1 - p) \alpha w(x) |\nabla v(x)|^{p} f(x)^{p} v(x)^{-p}.
\end{split}
\]
Then, Young's inequality $ab \le \frac{1}{p} a^{p} + \frac{p - 1}{p} b^{\frac{p}{p - 1}}$ ($a, b \ge 0$) yields
\[
\begin{split}
\A(x, \nabla v(x)) \cdot \nabla (f(x)^{p} v(x)^{1 - p})
\le
\beta^{p} \alpha^{1 - p} |\nabla f(x)|^{p} w(x).
\end{split}
\]
Therefore,
\[
\int_{\Omega} |f|^{p} \frac{d \sigma}{(u + \epsilon)^{p - 1}}
\le
\beta^{p} \alpha^{1 - p} \int_{\Omega} |\nabla f|^{p} \, d w,
\quad \forall f \in C_{c}^{\infty}(\Omega).
\]
The desired inequality follows from the monotone convergence theorem.
\end{proof}

\begin{proof}[Proof of Theorem \ref{thm:main_theorem}]
We first prove the existence part.
Take $\{ \sigma_{k} \}_{k = 1}^{\infty} \subset S_{c}(\Omega)$ such that
$\sigma_{k} = \mathbf{1}_{F_{k}} \sigma$ and $\mathbf{1}_{F_{k}} \uparrow \mathbf{1}_{\Omega}$ $\sigma$-a.e.
Set $u_{k} = \w_{\A} \sigma_{k} \in H_{0}^{1, p}(\Omega; w) \cap L^{\infty}(\Omega)$.
By \eqref{eqn:homogenity}, for each $\epsilon > 0$, we have
\[
\begin{split}
&
\frac{1}{q} \left( \frac{p - 1}{p - q} \right)^{p - 1}
\int_{\Omega} ( u_{k}^{\frac{q (p - 1)}{p - q}} - \epsilon^{\frac{q (p - 1)}{p - q}} )_{+} \, d \sigma_{k}
\\
& =
\int_{ \{ x \in \Omega \colon u_{k}(x) > \epsilon \} } \A(x, \nabla u_{k}^{\frac{p - 1}{p - q}} ) \cdot \nabla u_{k}^{\frac{p - 1}{p - q}}  \, dx.
\end{split}
\]
Take the limit $\epsilon \to 0$. By the monotone convergence theorem and \eqref{eqn:coercive},
\[
\begin{split}
\alpha
\int_{\Omega} |\nabla u_{k}^{\frac{p - 1}{p - q}}|^{p} \, d w
& \le
\int_{\Omega} \A(x, \nabla u_{k}^{\frac{p - 1}{p - q}} ) \cdot \nabla u_{k}^{\frac{p - 1}{p - q}} \, dx
\\
& =
\frac{1}{q} \left( \frac{p - 1}{p - q} \right)^{p - 1}
\int_{\Omega} u_{k}^{\frac{q (p - 1)}{p - q}} \, d \sigma_{k}.
\end{split}
\]
The right-hand side is finite, and thus, $u_{k}^{\frac{p - 1}{p - q}} \in H_{0}^{1, p}(\Omega; w)$.
By \eqref{eqn:trace_ineq} and density,
\[
\begin{split}
\int_{\Omega} u_{k}^{\frac{q (p - 1)}{p - q}} \, d \sigma_{k}
& \le
C_{1}^{q} \left( \int_{\Omega} |\nabla u_{k}^{\frac{p - 1}{p - q}}|^{p} \, d w \right)^{ \frac{q}{p} }.
\end{split}
\]
Combining the two inequalities, we obtain
\[
\begin{split}
\left( \int_{\Omega} |\nabla u_{k}^{\frac{p - 1}{p - q}}|^{p} \, d w \right)^{\frac{p - q}{p}}
\le
\frac{1}{q} \left( \frac{p - 1}{p - q} \right)^{p - 1} \frac{1}{\alpha} C_{1}^{q}.
\end{split}
\]
By the Poincar\'{e} inequality, $u_{k} \uparrow u \not \equiv \infty$,
so $u = \w_{\A} \sigma$ is $\A$-superharmonic in $\Omega$ and satisfies \eqref{eqn:p-laplace} by Theorem \ref{thm:TW} and the monotone convergence theorem.
By the uniqueness of the limit, $u^{\frac{p - 1}{p - q}}$ satisfies the latter inequality in \eqref{eqn:energy_bound}.
%Assume that $q \ge 1$. Without loss of generality, we may assume $\sigma_{k_{0}} \neq 0$ for some $k_{0}$.
%Fix a ball $B$ such that $B \Subset \Omega$.
%By the strong minimum principle, $\inf_{B} u_{k_{0}} > 0$.
%Then, for any $k \ge k_{0}$,
%\[
%c(p, q) \int_{\Omega} |\nabla u_{k}^{\frac{p - 1}{p - q}}|^{p} \, d w
%=
%\int_{\Omega} |\nabla u_{k}|^{p} u_{k}^{\frac{p (q- 1)}{p - q}} \, d w
%\ge
%\int_{B} |\nabla u_{k}|^{p} \, dw \cdot \inf_{B} u_{k_{0}}^{\frac{p (q- 1)}{p - q}}.
%\]
%On the other hand, by the Poincar\'{e} inequality,
%\[
%\left( \int_{B} u_{k}^{p} \, d w \right)^{\frac{p - 1}{p - q}}
%\le
%w(B)^{\frac{q - 1}{p - q}} \int_{\Omega} u_{k}^{\frac{p (p - 1)}{p - q}} \, d w
%\le
%C \int_{\Omega} |\nabla u_{k}^{\frac{p - 1}{p - q}}|^{p} \, d w.
%\]
%Hence $\| u_{k} \|_{H^{1, p}(B; w)}$ is bounded and $u \in H^{1, p}(B; w)$.
%Thus, $u \in H^{1, p}_{\loc}(\Omega; w)$.

Conversely, assume the existence of $u$.
Take $\{ \sigma_{k} \}_{k = 1}^{\infty} \subset S_{c}(\Omega)$ such that
$\sigma_{k} = \mathbf{1}_{F_{k}} \sigma$ and $\mathbf{1}_{F_{k}} \uparrow \mathbf{1}_{\Omega}$ $\sigma$-a.e. by using Theorem \ref{thm:approximation},
and set $u_{k} = \w_{\A} \sigma_{k}$.
By Lemma \ref{lem:picone},
\begin{equation}\label{eqn:lower_bound@main_theorem}
\begin{split}
\int_{\Omega} |f|^{q} \, d \sigma_{k}
& =
\int_{\Omega} |f|^{q} u_{k}^{p - 1} \frac{d \sigma_{k}}{u_{k}^{p - 1}}
\\
& \le
\left( \int_{\Omega} |f|^{p} \frac{d \sigma_{k}}{u_{k}^{p - 1}} \right)^\frac{q}{p}
\left( \int_{\Omega} u_{k}^{\frac{p(p - 1)}{p - q}} \frac{d \sigma_{k}}{u_{k}^{p - 1}} \right)^{\frac{p - q}{p}}
\\
& \le
\left( \beta^{p} \alpha^{1 - p} \int_{\Omega} |\nabla f|^{p} \, d w \right)^\frac{q}{p}
\left( \int_{\Omega} u_{k}^{\frac{q(p - 1)}{p - q}} \, d \sigma_{k} \right)^{\frac{p - q}{p}}.
\end{split}
\end{equation}
Note that $v := u^{\frac{p - 1}{p - q}} \ge u_{k}^{\frac{p - 1}{p - q}} $ by Theorem \ref{thm:comparison_principle}.
Since $v \in H_{0}^{1, p}(\Omega; w)$, by density,
\[
\begin{split}
\int_{\Omega} u_{k}^{\frac{q(p - 1)}{p - q}}  \, d \sigma_{k}
\le
\int_{\Omega} v^{q} \, d \sigma_{k}
\le
\left( \beta^{p} \alpha^{1 - p} \int_{\Omega} |\nabla v|^{p} \, d w \right)^\frac{q}{p}
\left( \int_{\Omega} u_{k}^{\frac{q(p - 1)}{p - q}}  \, d \sigma_{k} \right)^{\frac{p - q}{p}}.
\end{split}
\]
Using \eqref{eqn:lower_bound@main_theorem} again, we obtain
\[
\int_{\Omega} |f|^{q} \, d \sigma_{k}
\le
\beta^{p} \alpha^{1 - p}
\left( \int_{\Omega} |\nabla f|^{p} \, d w \right)^\frac{q}{p}
\left( \int_{\Omega} |\nabla v|^{p} \, d w \right)^\frac{p - q}{p}.
\]
The former inequality in \eqref{eqn:energy_bound} follows from the monotone convergence theorem.
\end{proof}

%If $C_{1}$ is the best constant of \eqref{eqn:trace_ineq}, then
%\begin{equation}\label{eqn:energy_bound}
%\frac{ (\alpha / \beta)^{p} }{\alpha} C_{1}^{q}
%\le
%\| \nabla u^{\frac{p - 1}{p - q}} \|_{L^{p}(\Omega; w)}^{p - q}
%\le
%\frac{1}{q} \left( \frac{p - 1}{p - q} \right)^{p - 1} \frac{1}{\alpha} C_{1}^{q}.
%\end{equation}
%The constants in \eqref{eqn:energy_bound} are independent of $n$ or $w$. %or other parameters.
%The assumption on $w$ leads to the regularity of solutions via the Sobolev inequality
%and is essential to justify the calculations in the proof.
%However, the form of \eqref{eqn:energy_bound} is derived directly from the structure conditions \eqref{eqn:coercive}-\eqref{eqn:homogenity}.
%Moreover, the upper bound in \eqref{eqn:energy_bound} is independent of $\beta$.
%We also have the following estimate:
%\begin{equation}\label{eqn:COV-bound}
%\frac{ (\alpha / \beta) }{ \alpha^{ \frac{1}{p} } }
%C_{1}
%\le
%\left( \int_{\Omega} (\w_{\A} \sigma)^{ \frac{q(p - 1)}{p - q} }  \, d \sigma \right)^{\frac{p - q}{qp}}
%\le
%\frac{1}{q^{ \frac{1}{p} }}
%\left( \frac{p - 1}{p - q} \right)^{\frac{p - 1}{p}}
%\frac{1}{ \alpha^{ \frac{1}{p} } }
%C_{1}.
%\end{equation}
%Thus,
If $\A$ and $\A'$ satisfy the same structure conditions, then
\[
\left( \int_{\Omega} (\w_{\A} \sigma)^{ \frac{q(p - 1)}{p - q} }  \, d \sigma \right)^{\frac{p - q}{qp}}
\approx_{\alpha, \beta}
\left( \int_{\Omega} (\w_{\A'} \sigma)^{ \frac{q(p - 1)}{p - q} }  \, d \sigma \right)^{\frac{p - q}{qp}}.
\]
Note that we do not have a global pointwise estimate between solutions.
The boundary behavior of two solutions may not be comparable.

Finally, we observe examples of $\sigma$ satisfying \eqref{eqn:trace_ineq}.
If $w \equiv 1$, $p < n$ and $\sigma \in L^{m}(\Omega; dx)$ with
\begin{equation}\label{eqn:exp_BO}
m = \left( \frac{p^{*}}{q} \right)' = \left( \frac{n p}{q(n - p)} \right)',
\end{equation}
then \eqref{eqn:trace_ineq} follows from Sobolev's inequality and H\"{o}lder's inequality.

Other type examples can be found from Hardy type inequalities.
For one-dimensional cases, we refer to \cite[Section 1.3.3]{MR2777530} and \cite{MR1395069}.
To present multi-dimensional sufficient conditions, we add an assumption to the boundary of $\Omega$.
We say that a bounded domain $\Omega$ is \textit{Lipschitz}
if for each $y \in \partial \Omega$,
there exist a local Cartesian coordinate system $(x_{1}, \dots, x_{n}) = (x', x_{n})$,
an open neighborhood $U = U_{y}$
and a Lipschitz function $h = h_{y}$ such that $\Omega \cap U = \{ (x', x_{n}) \colon x_{n} > h(x') \} \cap U$. 
If $\Omega$ is bounded and Lipschitz, then by the Hardy inequality in \cite[Theorem 1.6]{MR163054},
\begin{equation}\label{eqn:hardy}
\int_{\Omega} |f|^{p} \delta^{t - p} \, dx
\le
C \int_{\Omega} |\nabla f|^{p} \delta^{t} \, dx,
\quad
\forall f \in C_{c}^{\infty}(\Omega),
\end{equation}
where $\delta(x) := \dist(x, \partial \Omega)$, $t < p - 1$ and $C$ is a constant independent of $f$.
%The following existence theorem holds.

\begin{proposition}\label{prop:trace}
Let $\Omega$ be a bounded Lipschitz domain.
Set $w = \delta^{t}$ and $d \sigma = \delta^{- s} dx$, where $-1 < t < p - 1$ and $1 \le s \le p - t$.
Then, for any $q \in (\frac{p (s - 1)}{p - 1 - t}, p)$, the trace inequality \eqref{eqn:trace_ineq} holds.
\end{proposition}

\begin{proof}
By the inner cone property of Lipschitz graphs, for each $r > -1$,
\[
\int_{\Omega} \delta^{r} \, dx < \infty.
\]
(For further information about the integrability of $\delta$, see \cite[Theorem 6]{MR1668136} and the references therein.)
Thus, by the Hardy inequality \eqref{eqn:hardy},
\[
\begin{split}
\int_{\Omega} |f|^{q} \delta^{- s} \, dx
& =
\int_{\Omega} \frac{ |f|^{q} }{ \delta^{q} } \delta^{q - s - t} \delta^{t} \, dx
\\
& \le
\left(
\int_{\Omega} \frac{ |f|^{p} }{ \delta^{p} } \delta^{t} \, dx
\right)^{\frac{q}{p}}
\left(
\int_{\Omega}
\delta^{ \frac{p(q - s - t)}{p - q} }
\delta^{t} \, dx
\right)^{\frac{p - q}{p}}
\\
& \le
C \left( \int_{\Omega} |\nabla f|^{p} \delta^{t} \, dx \right)^{\frac{q}{p}},
\quad
\forall f \in C_{c}^{\infty}(\Omega).
\end{split}
\]
This completes the proof.
\end{proof}

\begin{remark}\label{rem:weight}
The boundary of a bounded Lipschitz domain is $(n - 1)$-regular.
Thus, by \cite[Lemma 3.3]{MR2606245} (see also \cite{MR3215609} and \cite[Chapter A.3]{MR2867756}), 
$w = \delta^{t}$ is an $A_{p}$-weight (and hence a $p$-admissible weight) on $\R^{n}$ for $-1 < t < p - 1$.
\end{remark}

\begin{corollary}\label{cor:sufficient_condition}
Let $\Omega$ be a bounded Lipschitz domain and let $w = \delta^{t}$, where $-1 < t < p - 1$.
Let $1 \le s < p - t$.
Assume that $\sigma \in L^{1}_{\loc}(\Omega)$ satisfies
 $0 \le \sigma(x) \le C \delta(x)^{- s}$ for a.e. $x \in \Omega$.
Then, there exists a minimal nonnegative $\A$-superharmonic solution $u$ to \eqref{eqn:p-laplace}.
\end{corollary}

\begin{proof}
Combine Proposition \ref{prop:trace}, Remark \ref{rem:weight} and Theorem \ref{thm:main_theorem}.
\end{proof}

\begin{remark}
Since the boundary of a Lipschitz domain may not be smooth,
Corollary \ref{cor:sufficient_condition} is not trivial even if $\divergence \A(x, \nabla u) = \laplacian u$.
In fact, the Green function of a polygon is not comparable to $\delta$ near the corners. %(see, e.g., \cite{MR510197}).
This result asserts that the threshold of $s$ still be $2 \, (= p)$ even if such a case.
Ancona \cite{MR856511} proved a more general existence theorem for (unweighted) linear elliptic equations by using strong barriers.
The author is unaware of extensions of his results to nonlinear equations.
\end{remark}

\begin{remark}
Note that $L^{p}(w)$-$L^{p}(\sigma)$ trace inequalities do not yield the existence of (bounded) solutions to \eqref{eqn:p-laplace}.
Let $\Omega = B(0, 1)$ and $d \sigma = \delta^{-2} \, dx$.
Then, the $L^{2}(dx)$-$L^{2}(d \sigma)$ trace inequality holds by \eqref{eqn:hardy},
but the Green potential of $\sigma$ does not exist.
\end{remark}

%%%%%%%%%%%%%%%%%%%%%%%%%%%%%%%%%%%%%%%%
\section{Weak-type inequality and capacitary condition}\label{sec:weak-trace}
%%%%%%%%%%%%%%%%%%%%%%%%%%%%%%%%%%%%%%%%

%\subsection{Weak-type trace inequality}

In \cite{MR1734322}, the following form of a weak-type trace inequality was studied:
\begin{equation}\label{eqn:weak_trace@prop:tr-wni-cap}
\| f \|_{L^{q, \infty}(\sigma)}
\le
C_{2} \| \nabla f \|_{L^{p}(w)},
\quad
\forall f \in C_{c}^{\infty}(\Omega).
\end{equation}
Here, $L^{q, \infty}(\sigma)$ is the Lorentz space with respect to $\sigma$ (see, e.g., \cite[Chapter 1]{MR2445437}).
By a truncation argument, \eqref{eqn:weak_trace@prop:tr-wni-cap} implies 
\[
\sigma(K)^{\frac{p}{q}} \le C_{2}^{p} \, \capacity_{p, w}(K, \Omega), \quad \forall K \Subset \Omega.
\]
Hence, by Maz'ya's capacitary inequality (see, e.g., \cite{MR2146047}, \cite[Lemma 6.22]{MR2867756}), 
\eqref{eqn:weak_trace@prop:tr-wni-cap} is equivalent to the embedding into $L^{q, p}(\sigma)$.
In particular, the condition \eqref{eqn:weak_trace@prop:tr-wni-cap} is weaker than \eqref{eqn:trace_ineq}
because $L^{q}(\sigma) = L^{q, q}(\sigma) \subsetneq L^{q, p}(\sigma)$ for $q < p$.

%Quin and Verbitsky \cite{MR3724493} studied related weighted norm inequalities of linear operators.
%We present their counterparts in nonlinear potential theory.

\begin{theorem}\label{thm;weak-COV}
Assume that $0 < q < p$.
Let $\sigma \in \M_{0}^{+}(\Omega)$ and let $C_{2}$ be the best constant of \eqref{eqn:weak_trace@prop:tr-wni-cap}.
Then,
\begin{equation}\label{eqn:COV@weak-COV}
\frac{ (\alpha / \beta) }{ \alpha^{ \frac{1}{p} } }
C_{2}
\le
\| \w_{\A} \sigma \|_{L^{\frac{q(p - 1)}{p - q}, \infty}(\sigma)}^{\frac{p - 1}{p}}
\le
4^{\frac{p - 1}{p - q}} 
\frac{1}{ \alpha^{ \frac{1}{p} } }
C_{2}.
\end{equation}
\end{theorem}

To prove the lower bound in \eqref{eqn:COV@weak-COV},
we consider the counterparts of Lemma 5.10 and Proposition 6.1 in \cite{MR3724493}.

\begin{lemma}\label{lem:weak-wni}
Let $0 < q < \infty$.
Assume that $\sigma \in \M_{0}^{+}(\Omega)$ satisfies the following  weighted norm inequality:
\begin{equation}\label{eqn:WNI@prop:tr-wni-cap}
\| \w_{\A} \nu \|_{L^{\frac{q(p- 1)}{p}, \infty}(\sigma)}^{p - 1}
\le
C_{2}' \nu(\Omega),
\quad \forall \nu \in \M^{+}_{0}(\Omega).
\end{equation}
Then, \eqref{eqn:weak_trace@prop:tr-wni-cap} holds with
$C_{2} = \left( \beta^{p} \alpha^{1 - p} C_{2}' \right)^{ \frac{1}{p} }$.
\end{lemma}

\begin{proof}
Let $K$ be any compact subset of $\Omega$.
Let  $u$ be the $\A$-potential of a condenser $(K, \Omega)$ and let $\nu$ be the Riesz measure of $u$.
Then, by \cite[Corollary 4.8]{MR1386213}, %(see also \cite[Lemma 3.7]{MR1264000}),
\[
\nu(\Omega)
\le
\beta^{p} \alpha^{1 - p}
\capacity_{p, w}(K, \Omega).
\]
Since $u = \w_{\A} \nu \ge 1$ on $K$,
\[
\sigma(K)^{\frac{p}{q}}
\le
\| \w_{\A} \nu \|_{L^{\frac{q(p- 1)}{p}, \infty}(\sigma)}^{p - 1}
\le
C_{2}' \nu(\Omega)
\le
\beta^{p} \alpha^{1 - p}
C_{2}' \capacity_{p, w}(K, \Omega).
\]
Therefore,
\[
\begin{split}
\| f \|_{L^{q, \infty}(\sigma)}^{p}
& =
\sup_{k \ge 0} k^{p} \, \sigma( \{ |f| \ge k \} )^{\frac{p}{q}}
\\
& \le
C_{2}^{p} \sup_{k \ge 0} k^{p} \, \capacity_{p, w}(\{ |f| \ge k \}, \Omega)
\le
C_{2}^{p} \| \nabla f \|_{L^{p}(w)}^{p}.
\end{split}
\]
This completes the proof
\end{proof}

\begin{lemma}\label{lem:derivative}
Let $\nu, \sigma \in \M_{0}^{+}(\Omega)$.
Then,
\[
\left\| \frac{\w_{\A} \nu}{ \w_{\A} \sigma} \right\|_{L^{p - 1, \infty}(\sigma)} \le \nu(\Omega)^{\frac{1}{p - 1}}.
\]
\end{lemma}

\begin{proof}
Without loss of generality, we may assume that $\sigma, \nu \in S_{c}(\Omega)$.
Set $u = \w_{\A} \sigma$ and $v = \w_{\A} \nu_{t}$, where $\nu_{t} = t^{1 - p} \nu$.
For $k > 0$, set $I_{k}( v - u ) = k^{-1} \min\{ (v - u)_{+} , k \}$.
By \eqref{eqn:monotonicity},
\[
\begin{split}
&
\int_{\Omega} I_{k}(v - u) \, d \nu_{t} - \int_{\Omega} I_{k}(v - u) \, d \sigma
\\
& =
\frac{1}{k} \int_{ \{ x \in \Omega: 0 < v(x) - u(x) < k \} } \left( \A(x, \nabla v) - \A(x, \nabla u) \right) \cdot \nabla (v - u) \, dx
\ge
0.
\end{split}
\]
Passing to the limit $k \to 0$ yields
\[
\nu_{t}(\Omega) \ge \sigma( \{ x\in \Omega: v(x) > u(x) \} ).
\]
By \eqref{eqn:homogenity}, $v = t^{-1} \w_{\A} \nu$, and thus,
\[
t^{p - 1} \sigma\left( \left\{ x \in \Omega: \frac{ \w_{\A} \nu(x) }{ \w_{\A} \sigma(x) } > t \right\} \right) \le \nu(\Omega).
\]
Taking the supremum over $t > 0$, we obtain the desired inequality.
\end{proof}

%\subsection{Non-capacitary characterization}

\begin{proof}[Proof of Theorem \ref{thm;weak-COV}]
We first prove the latter inequality for $\sigma \in S_{c}(\Omega)$.
Let $u = \w_{\A} \sigma$.
Then,
\[
\sup_{j \in \Z} 2^{j \frac{q(p - 1)}{p - q} } \sigma(E_{j})
\le
\| u \|_{L^{\frac{q(p - 1)}{p - q}, \infty}(\sigma)}^{ \frac{q(p - 1)}{p - q} }
\le
\sup_{j \in \Z} 2^{(j + 1) \frac{q(p - 1)}{p - q} } \sigma(E_{j}),
\]
where $E_{j} = \{ x \in \Omega \colon u(x) > 2^{j} \}$.
By \eqref{eqn:weak_trace@prop:tr-wni-cap}, we have
\[
\sigma(E_{j})
\le
C_{2}^{q} \capacity_{p, w}(E_{j}, \Omega)^{\frac{q}{p}}
\le
C_{2}^{q} \capacity_{p, w}(E_{j}, E_{j - 1})^{\frac{q}{p}}.
\]
Let $U_{j} =  \min\{ (u - 2^{j - 1})_{+}, 2^{j - 1} \}$.
Then $U_{j} \in H_{0}^{1, p}(E_{j - 1}; w)$, $0 \le U_{j} \le 2^{j - 1}$ in $E_{j - 1}$ and $U_{j} \equiv 2^{j - 1}$ on $E_{j}$.
By \eqref{eqn:coercive},
\[
\begin{split}
\capacity_{p, w}(E_{j}, E_{j - 1})
& \le
2^{(1 - j)p} \int_{\Omega} |\nabla U_{j}|^{p} \, dw
\le
\frac{2^{(1 - j)p}}{\alpha} \int_{\Omega} \A(x, \nabla u) \cdot \nabla U_{j} \, dx
\\
& =
\frac{2^{(1 - j)p}}{\alpha} \int_{\Omega} U_{j} \, d \sigma
\le
\frac{2^{p - 1} 2^{j(1 - p)}}{\alpha} \sigma(E_{j - 1}).
\end{split}
\]
Combining these inequalities, we obtain
\[
\begin{split}
\| u \|_{L^{\frac{q(p - 1)}{p - q}, \infty}(\sigma)}^{ \frac{q(p - 1)}{p - q} }
& \le
2^{\frac{q(p - 1)}{p - q}}
\frac{ C_{2}^{q} }{ \alpha^{ \frac{q}{p} } }
\left(
2^{p - 1}
\sup_{j \in \Z} 2^{j \frac{q(p - 1)}{p - q} } \sigma(E_{j - 1})
\right)^{\frac{q}{p}}
\\
& \le
4^{\frac{q(p - 1)}{p - q}}
\frac{ C_{2}^{q} }{ \alpha^{ \frac{q}{p} } }
\left(
\| u \|_{L^{\frac{q(p - 1)}{p - q}, \infty}(\sigma)}^{ \frac{q(p - 1)}{p - q} }
\right)^{\frac{q}{p}}.
\end{split}
\]
Hence the desired inequality holds.

Let us prove the existence.
Take $\{ \sigma_{k} \}_{k = 1}^{\infty} \subset S_{c}(\Omega)$ such that
$\sigma_{k} = \mathbf{1}_{F_{k}} \sigma$ and $\mathbf{1}_{F_{k}} \uparrow \mathbf{1}_{\Omega}$ $\sigma$-a.e.
By the monotone convergence theorem,
\[
\| \w_{\A} \sigma \|_{L^{\frac{q(p - 1)}{p - q}, \infty}(\sigma)}^{\frac{q(p - 1)}{p}}
\le
\lim_{k \to \infty}
\| \w_{\A} \sigma_{k} \mathbf{1}_{F_{k}} \|_{L^{\frac{q(p - 1)}{p - q}, \infty}(\sigma)}^{\frac{q(p - 1)}{p}}
\le
\frac{ 4^{\frac{q(p - 1)}{p - q}} }{ \alpha^{ \frac{q}{p} } }
C_{2}^{q}.
\]
Thus $\w_{\A} \sigma \not \equiv \infty$ and is $\A$-superharmonic in $\Omega$.

Conversely, assume that $\| \w_{\A} \sigma \|_{L^{\frac{q(p - 1)}{p - q}, \infty}(\sigma)}$ is finite.
By H\"{o}lder's inequality for Lorentz spaces, Lemma \ref{lem:derivative} yields
\[
\begin{split}
\| \w_{\A} \nu \|_{L^{\frac{q(p- 1)}{p}, \infty}(\sigma)}
& \le
\| \w_{\A} \sigma \|_{L^{\frac{q(p - 1)}{p - q}, \infty}(\sigma)}
\left\| \frac{\w_{\A} \nu}{ \w_{\A} \sigma} \right\|_{L^{p - 1, \infty}(\sigma)}
\\
& \le
\| \w_{\A} \sigma \|_{L^{\frac{q(p - 1)}{p - q}, \infty}(\sigma)} \nu(\Omega)^{\frac{1}{p - 1}},
\quad
\forall \nu \in \M_{0}^{+}(\Omega).
\end{split}
\]
From Lemma \ref{lem:weak-wni}, the desired lower bound follows.
\end{proof}

%\begin{remark}
%As asserted in \cite{MR1734322}, 
%if there exists a bounded (super)solution $u$ to \eqref{eqn:p-laplace},
%then the $L^{p}(w)$-$L^{p}(\sigma)$ trace inequality holds.
%In fact,
%\[
%\begin{split}
%\| \w_{\A} \nu \|_{L^{p - 1, \infty}(\sigma)}
%& \le
%\left\| \w_{\A} \sigma \right\|_{L^{\infty}(\sigma)}
%\left\| \frac{\w_{\A} \nu}{\w_{\A} \sigma} \right\|_{L^{p - 1, \infty}(\sigma)}
%\\
%& \le
%\left\| u - \inf_{\Omega} u \right\|_{L^{\infty}(\sigma)} \nu(\Omega)^{\frac{1}{p - 1}},
%\quad
%\forall \nu \in \M_{0}^{+}(\Omega).
%\end{split}
%\]
%%and thus \eqref{eqn:CC@prop:tr-wni-cap} holds. %Since $q \geq p$, 
%Thus, \eqref{eqn:trace_ineq} follows from \cite[Corollary 11.3]{MR2777530}.
%
%By contrast, 
%$L^{p}(w)$-$L^{p}(\sigma)$ trace inequalities do not yield the existence of bounded solutions to \eqref{eqn:p-laplace}.
%Let $\Omega = B(0, 1)$ and $d \sigma = \dist(x, \partial \Omega)^{-2} \, dx$.
%Then, the $L^{2}(dx)$-$L^{2}(d \sigma)$ trace inequality holds by \eqref{eqn:hardy},
%but the Green potential of $\sigma$ does not exist.
%Furthermore, even if a solution $u$ to \eqref{eqn:p-laplace} exists, it may not be bounded.
%Let $1 < p < n$, $\Omega = B(0, 1) \subset \R^{n}$ and $d \sigma = |x|^{-p} dx$.
%Then, the $L^{p}(dx)$-$L^{q}(d \sigma)$ trace inequality holds for any $0 < q \le p$.
%However, $u = c \log |x|^{-1}$ is not bounded.
%Note that the constants in Theorems \ref{thm:main_theorem} and \ref{thm;weak-COV} are not bounded when $q \to p_{-}$.
%\end{remark}

%%%%%%%%%%%%%%%%%%%%%%%%%%%%%%%%%%%%%%%%
%\section{Relaxed capacitary condition}
%%%%%%%%%%%%%%%%%%%%%%%%%%%%%%%%%%%%%%%%

To treat the more general $\sigma \in \M_{0}^{+}(\Omega)$, we consider the following relaxed capacitary condition.
A typical example of such a measure is
the sum of a finite measure in $\M_{0}^{+}(\Omega)$ and a measure satisfying \eqref{eqn:trace_ineq}.
Note that it is not clear whether $\w_{\A}( \sigma_{1} + \sigma_{2}) \not \equiv \infty$
even if $\w_{\A} \sigma_{i} \not \equiv \infty$ for $i = 1, 2$.

\begin{proposition}\label{thm:relaxed_cc}
Let $\sigma \in \M_{0}(\Omega)$.
Assume that there exists a constant $C_{3} > 0$ such that
\begin{equation}\label{eqn:relaxed_cc}
\sigma(K) \le C_{3} \left( \capacity_{p, w}(K, \Omega)^{ \frac{q}{p} } + 1 \right), \quad \forall K \Subset \Omega,
\end{equation}
where $0 < q < p$ and $C_{3} > 0$ is a constant.
Then, there exists a minimal nonnegative $\A$-superharmonic solution $u$ to \eqref{eqn:p-laplace}.
\end{proposition}

\begin{proof}

We first claim that if $\sigma \in S_{c}(\Omega)$ satisfies \eqref{eqn:relaxed_cc}, then
\begin{equation}\label{eqn:bound@relaxed_cc}
\| (u - 1)_{+} \|_{L^{ p - 1, \infty }(w)} 
\le
C,
\end{equation}
where $C = C(p, q, w, \alpha, C_{3})$.
Let $E_{j} = \{ x \in \Omega \colon u(x) > 2^{j} \}$.
As in the proof of Theorem \ref{thm;weak-COV},
\[
\begin{split}
\sigma(E_{j})
\le
C \capacity_{p, w}(E_{j}, \Omega)^{ \frac{q}{p} } + C
\le
C \left( 2^{j (1 - p)}\sigma(E_{j - 1}) \right)^{ \frac{q}{p} } + C.
\end{split}
\]
Thus,
\[
2^{j \frac{q(p - 1)}{p - q}} \sigma(E_{j})
\le
C \left( 2^{(j - 1) \frac{q(p - 1)}{p - q}} \sigma(E_{j - 1}) \right)^{ \frac{q}{p} } + 2^{j \frac{q(p - 1)}{p - q}} C,
\]
and hence
\[
\begin{split}
&
\max\left\{ \sigma( E_{0} ), \sup_{ j \le -1 } 2^{j \frac{q(p - 1)}{p - q}} \sigma(E_{j}) \right\}
\le
C \left( \sup_{j \le -1} 2^{j \frac{q(p - 1)}{p - q}} \sigma(E_{j}) \right)^{ \frac{q}{p} } + C.
\end{split}
\]
By Young's inequality, $\sigma(E_{0}) \le C$.
Using the test function $\min\{ (u - 1)_{+}, l \}$ ($l > 0$), we obtain
\[
\begin{split}
\alpha \| \nabla \min\{ (u - 1)_{+}, l \} \|_{L^{p}(w)}^{p}
& \le
\int_{\Omega} \A(x, \nabla u) \cdot \nabla \min\{ (u - 1)_{+}, l \} \, dx
\\
& =
\int_{\Omega} \min\{ (u - 1)_{+}, l \} \, d \sigma
\le
l \sigma(E_{0}).
\end{split}
\]
Then, the Poincar\'{e} inequality yields
\[
l^{p} w( \{ x \in \Omega: (u - 1)_{+}(x) \le l \} ) \le C l.
\]
This result implies \eqref{eqn:bound@relaxed_cc}.

Take $\{ \sigma_{k} \}_{k = 1}^{\infty} \subset S_{c}(\Omega)$ such that
$\sigma_{k} = \mathbf{1}_{F_{k}} \sigma$ and $\mathbf{1}_{F_{k}} \uparrow \mathbf{1}_{\Omega}$ $\sigma$-a.e.
Then $\w_{A} \sigma_{k} \uparrow \w_{A} \sigma \not \equiv \infty$ by \eqref{eqn:bound@relaxed_cc}.
Hence, $u := \w_{A} \sigma$ is the desired minimal $\A$-superharmonic solution.
\end{proof}

%%%%%%%%%%%%%%%%%%%%%%%%%%%%%%%%%%%%%%%%
\section{Applications to singular elliptic problems}\label{sec:app}
%%%%%%%%%%%%%%%%%%%%%%%%%%%%%%%%%%%%%%%%

In this section, we consider Eq. \eqref{eqn:singular}.
For basics of singular elliptic problems, we refer to
\cite{MR427826, MR1037213, MR1458503, MR1866062, MR2592976, MR4152210} and the references therein.
We first prove the following general existence result.

\begin{theorem}\label{thm:singular}
Let $\sigma \in \M^{+}_{0}(\Omega) \setminus \{ 0 \}$.
Assume that there exists a nonnegative $\A$-superharmonic supersolution $v$ to \eqref{eqn:p-laplace}.
Let $h \colon (0, \infty) \to (0, \infty)$ be a continuously differentiable nonincreasing function.
Then there exists a nonnegative $\A$-superharmonic solution $u$ to \eqref{eqn:singular}
satisfying the Dirichlet boundary condition in the sense that
\begin{equation}\label{eqn:bound_sing}
0 < g(u)(x) \le v(x), \quad \forall x \in \Omega,
\end{equation}
where
\[
g(u) := \int_{0}^{u} \frac{1}{h(t)^{ \frac{1}{p - 1} } } \, dt.
\]
\end{theorem}

\begin{remark}
By assumption, $g$ is a convex increasing function.
In particular, $\lim_{t \to \infty} g(t) = \infty$.
Use of this type transformation can be found in \cite{MR1490771, MR1866062, MR3829750}.
\end{remark}

To prove Theorem \ref{thm:singular}, we use the following approximating problems:
\begin{equation}\label{eqn:approximate}
\begin{cases}
\displaystyle
- \divergence \A(x, \nabla u_{k}) = \sigma_{k} h \left(u_{k} + \frac{1}{k} \right) & \text{in} \ \Omega, \\
u = 0 & \text{on} \ \partial \Omega,
\end{cases}
\end{equation}
where $k \in \N$, $\sigma_{k} := \mathbf{1}_{F_{k}} \sigma$ and $\{ F_{k} \}_{k = 1}^{\infty}$ is a sequence of compact sets in Theorem \ref{thm:approximation}.
We may assume that $\sigma_{1} \neq 0$ without loss of generality.

If $\{ u_{k} \}_{k = 1}^{\infty} \subset H_{0}^{1, p}(\Omega; w)$ is a sequence of weak solutions to \eqref{eqn:approximate},
then $u_{k + 1} \ge u_{k}$ a.e. in $\Omega$ for all $k \ge 1$. In fact, using the test function $(u_{k} - u_{k + 1})_{+} \in H_{0}^{1, p}(\Omega; w)$, we have
\[
\begin{split}
&
\int_{\Omega} \left( \A(x, \nabla u_{k}) - \A(x, \nabla u_{k + 1}) \right)  \cdot \nabla (u_{k} - u_{k + 1})_{+} \, dx
\\
& =
\int_{\Omega} (u_{k} - u_{k + 1})_{+} \, h\left( u_{k} + \frac{1}{k} \right)  \, d \sigma_{k}
-
\int_{\Omega} (u_{k} - u_{k + 1})_{+} \, h\left( u_{k} + \frac{1}{k + 1} \right)  \, d \sigma_{k + 1}
\\
& \le
\int_{\Omega} (u_{k} - u_{k + 1})_{+} \left\{  h\left( u_{k} + \frac{1}{k + 1} \right) - h\left( u_{k + 1} + \frac{1}{k + 1} \right) \right\} \, d \sigma_{k + 1}
\le 0.
\end{split}
\]
By \eqref{eqn:monotonicity}, $\nabla u_{k} = \nabla u_{k + 1}$ a.e. in $\{ x \in \Omega \colon u_{k}(x) > u_{k + 1}(x) \}$,
and thus, $(u_{k} - u_{k + 1})_{+} = 0$ in $H_{0}^{1, p}(\Omega; w)$.

It is well-known that the singular problem has a convex structure,
so we use the Minty-Browder theorem (see, e.g., \cite[Corollary III.1.8]{MR567696} and \cite{MR212631}).
The same approach can be found in \cite{MR3579130}.

\begin{lemma}
There exists a nonnegative weak solution $u_{k} \in H_{0}^{1, p}(\Omega; w)$ to \eqref{eqn:approximate}.
\end{lemma}

\begin{proof}
Set $V = H_{0}^{1, p}(\Omega; w)$. For $u \in V$, we define
\[
A(u) := - \divergence \A(x, \nabla u) - \sigma_{k} h \left( u_{+} + \frac{1}{k} \right).
\]
We apply the Minty-Browder theorem to $A$.
Since $\sigma_{k} \in S_{c}(\Omega) \subset (H_{0}^{1, p}(\Omega))^{*}$,
\[
\begin{split}
\left| \int_{\Omega} \varphi \, h \left( u_{+} + \frac{1}{k} \right) \, d \sigma_{k} \right|
& \le
h \left( \frac{1}{k} \right) \int_{\Omega} |\varphi| \, d \sigma_{k}
\le
h \left( \frac{1}{k} \right) \| \sigma_{k} \|_{V^{*}} \| \nabla \varphi \|_{L^{p}(\Omega)}
\end{split}
\]
for all $\varphi \in V$.
Thus, $A$ is a bounded operator from $V$ to the dual $V^{*}$ of $V$.
Moreover,
\[
\frac{\langle A(u), u \rangle}{\| u \|_{V}} \to \infty
\quad \text{as} \quad \| u \|_{V} \to \infty.
\]
Since $h$ is nonincreasing,
\[
\int_{\Omega} (u - v) \, \left\{ h\left(u_{+} - \frac{1}{k} \right) - h\left(v_{+} - \frac{1}{k} \right) \right\} \, d \sigma_{k} \le 0, \quad \forall u, v \in V,
\]
and hence,
\[
\langle A(u) - A(v), u - v \rangle \ge 0, \quad \forall u, v \in V.
\]
Finally, we claim that if $\{ u_{i} \}_{i = 1}^{\infty} \subset V$ and $u_{i} \to u$ in $V$,
then for any $\varphi \in V$,
\begin{equation}\label{eqn:conv1}
\int_{\Omega} \A(x, \nabla u_{i}) \cdot \nabla \varphi \, dx \to \int_{\Omega} \A(x, \nabla u) \cdot \nabla \varphi \, dx 
\end{equation}
and
\begin{equation}\label{eqn:conv2}
\int_{\Omega} \varphi \, h \left( (u_{i})_{+} + \frac{1}{k} \right) \, d \sigma_{k}
\to
\int_{\Omega} \varphi \, h \left( u_{+} + \frac{1}{k} \right) \, d \sigma_{k}.
\end{equation}
The proof of \eqref{eqn:conv1} is standard (see \cite[Proposition 17.2]{MR2305115}).
Let $\{ u_{i_{j}} \}_{j = 1}^{\infty}$ be any subsequence of $\{ u_{i} \}_{i = 1}^{\infty}$.
Since $\sigma_{k} \in S_{c}(\Omega)$, the embedding $V \hookrightarrow L^{1}(\Omega; \sigma_{k})$ is continuous,
and hence $u_{i_{j}} \to u$ in $L^{1}(\Omega; \sigma_{k})$.
We choose a subsequence $\{ u_{i_{j'}} \}_{j' = 1}^{\infty}$ of $\{ u_{i_{j}} \}_{j = 1}^{\infty}$ such that
$u_{i_{j'}} \to u$ $\sigma_{k}$-a.e.
Since $h$ is continuous and nonincreasing, by the dominated convergence theorem, 
\[
\int_{\Omega} \varphi \, h \left( (u_{i_{j'}})_{+} + \frac{1}{k} \right) \, d \sigma_{k}
\to
\int_{\Omega} \varphi \, h \left( u_{+} + \frac{1}{k} \right) \, d \sigma_{k}.
\]
The right hand side is independent of the choice of $\{ u_{i_{j}} \}_{j = 1}^{\infty}$, and hence \eqref{eqn:conv2} holds.
Consequently, the map $A \colon V \to V^{*}$ is onto.
In particular, there exists a unique $u \in V$ such that $A(u) = 0$.
Since $u$ is a supersolution to $- \divergence \A(x, \nabla u) = 0$ in $\Omega$, $u \ge 0$ q.e. in $\Omega$.
Thus, $u = u_{+}$ $\sigma_{k}$-a.e. in $\Omega$ and satisfies \eqref{eqn:approximate}.
\end{proof}

\begin{lemma}\label{lem:bound_for_approximate_sol}
Let $u_{k} \in H_{0}^{1, p}(\Omega; w)$ be an lsc-regularized weak solution to \eqref{eqn:approximate}.
Then $0 < g(u_{k}) \le \w_{\A} \sigma_{k}$ in $\Omega$.
\end{lemma}

\begin{proof}
By the comparison principle for weak solutions,
\[
0 \le u_{k}(x) \le \w_{\A} \left( h \left( \frac{1}{k} \right) \sigma_{k} \right)(x), \quad \forall x \in \Omega.
\]
Since $\sigma_{k} \in S_{c}(\Omega)$, $u_{k} \in H_{0}^{1, p}(\Omega; w) \cap L^{\infty}(\Omega)$.
Fix a nonnegative function $\varphi \in C_{c}^{\infty}(\Omega)$.
By assumption, the function $t \mapsto \left( \frac{1}{h(t)} - \epsilon \right)_{+} $ ($\epsilon > 0$) is nondecreasing and locally Lipschitz.
Consider the test function $\left( \frac{1}{h(u_{k})} - \epsilon \right)_{+} \varphi \in H_{0}^{1, p}(\Omega; w)$.
Since
\[
\int_{\Omega} \A(x, \nabla u_{k}) \cdot \nabla \left( \frac{1}{h(u_{k})} - \epsilon \right)_{+} \varphi \, dx \ge 0,
\]
we have
\[
\begin{split}
&
\int_{\Omega} \A(x, \nabla u_{k}) \cdot \nabla \varphi \, \left( \frac{1}{h(u_{k})} - \epsilon \right)_{+} \, dx
\\
& \le
\int_{\Omega}
\left( \frac{1}{h(u_{k})} - \epsilon \right)_{+} \varphi \, h \left( u_{k} + \frac{1}{k} \right)
\, d \sigma_{k}
\le
\int_{\Omega} \varphi \, d \sigma_{k}.
\end{split}
\]
By \eqref{eqn:homogenity} and the dominated convergence theorem,
\[
\begin{split}
\int_{\Omega} \A(x, \nabla g(u_{k})) \cdot \nabla \varphi \, dx
=
\int_{\Omega} \A(x, \nabla u_{k}) \cdot \nabla \varphi \, \frac{1}{ h(u_{k}) } \, dx
\le
\int_{\Omega} \varphi \, d \sigma_{k}.
\end{split}
\]
By the comparison principle for weak solutions, $0 < g(u_{k}) \le \w_{\A} \sigma_{k}$ a.e. in $\Omega$.
In other words,  $0 <u_{k} \le g^{-1}( \w_{\A} \sigma_{k} )$ a.e. in $\Omega$, where $g^{-1}$ is the inverse function of $g$.
Since $g^{-1}$ is concave and increasing, $g^{-1}( \w_{\A} \sigma_{k} )$ is $\A$-superharmonic in $\Omega$,
and thus, the same inequality holds for all $x \in \Omega$.
\end{proof}

\begin{proof}[Proof of Theorem \ref{thm:singular}]
Let $\{ u_{k} \}_{k = 1}^{\infty}$ be the sequence of lsc-regularized weak solutions to \eqref{eqn:approximate}.
By Lemma \ref{lem:bound_for_approximate_sol},
\[
0 < g( u_{k} ) \le \w_{\A} \sigma_{k} \le \w_{\A} \sigma \le v
\quad \text{in} \ \Omega.
\]
Thus, $u(x) := \lim_{k \to \infty} u_{k}(x)$ is not identically infinite.
By Theorem \ref{thm:TW}, $\mu[u_{k}]$ converges to $\mu[u]$ weakly.
Fix $\varphi \in C_{c}^{\infty}(\Omega)$.
By the weak Harnack inequality, 
there exists a constant $c$ such that $u_{k} \ge u_{1} \ge c > 0$ on $\spt \varphi$.
By the dominated convergence theorem,
\[
\lim_{k \to \infty}
\int_{\Omega} \varphi \, h \left(u_{k} + \frac{1}{k} \right) \, d \sigma_{k}
=
\int_{\Omega} \varphi \, h(u) \, d \sigma.
\]
Thus, $u$ satisfies \eqref{eqn:singular} in the sense of $\A$-superharmonic solutions.
\end{proof}

The following existence result was established by Boccardo and Orsina \cite{MR2592976} for unweighted equations.
The necessity part seems to be new.

\begin{corollary}\label{cor:BO}
Suppose that $\sigma \in \M^{+}_{0}(\Omega) \setminus \{ 0 \}$.
Set $h(u) = u^{- \gamma}$, where $0 < \gamma < \infty$.
Then there exists an $\A$-superharmonic solution $u$ to \eqref{eqn:singular}
satisfying $u^{\frac{p - 1 + \gamma}{p}} \in H_{0}^{1, p}(\Omega; w)$
if and only if $\sigma$ is finite.
\end{corollary}

\begin{proof}
Assume that $\sigma$ is finite. 
Then $\w_{\A} \sigma \not \equiv \infty$ by \cite[Theorem 6.6]{MR1386213}.
Thus, the existence of $u = \w_{A}(u^{- \gamma} \sigma)$ follows from Theorem \ref{thm:singular}.
Set $\mu = u^{- \gamma} \sigma$. 
By \eqref{eqn:energy_bound} and \eqref{eqn:COV-bound},
\[
\int_{\Omega} \, d \sigma 
=
\int_{\Omega} u^{\gamma} u^{- \gamma} \, d \sigma
=
\int_{\Omega} \left( \w_{\A} \mu \right)^{\gamma} \, d \mu
\approx_{\alpha, \beta}
\int_{\Omega} |\nabla u^{ \frac{p - 1 + \gamma}{p} } |^{p} \, d w.
\]
The necessity part follows from this two-sided estimate directly.
\end{proof}

\begin{remark}\label{rem:uniqueness}
If there exists a finite energy solution $u \in H_{0}^{1, p}(\Omega; w)$ to \eqref{eqn:singular}, then
\begin{equation}\label{eqn:embedding_sing}
\begin{split}
\int_{\Omega} \varphi \, h(u) \, d \sigma
& =
\int_{\Omega} \A(x, \nabla u) \cdot \nabla \varphi \, dx
\\
& \le
\beta \left( \int_{\Omega} |\nabla u|^{p} \, d w \right)^{\frac{p - 1}{p}}
\left( \int_{\Omega} |\nabla \varphi|^{p} \, d w \right)^{\frac{1}{p}},
\quad \forall \varphi \in C_{c}^{\infty}(\Omega).
\end{split}
\end{equation}
Thus, the embedding $H_{0}^{1, p}(\Omega; w) \hookrightarrow L^{1}(\Omega; h(u) \, \sigma)$ is continuous.
Furthermore, such a solution is unique in $H_{0}^{1, p}(\Omega; w)$.
In fact, if $v \in H_{0}^{1, p}(\Omega; w)$ is another solution, then
\[
\int_{\Omega} \left( \A(x, \nabla u) - \A(x, \nabla v) \right)  \cdot \nabla (u - v) \, dx
=
\int_{\Omega} (u - v) \left( h(u) -h(v) \right) \, d \sigma \le 0.
\]
Hence $\nabla u = \nabla v$ a.e. in $\Omega$ and $u = v$ in $H_{0}^{1, p}(\Omega; w)$.
\end{remark}

\begin{proof}[Proof of Theorem \ref{thm:singular_fe}]
Assume that \eqref{eqn:trace_ineq} holds.
Let $h(u) = u^{q - 1}$, and let $\{ u_{k} \}_{k = 1}^{\infty}$ be the sequence of lsc-regularized weak solutions to \eqref{eqn:approximate}.
Then
\[
\begin{split}
\alpha \int_{\Omega} |\nabla u_{k}|^{p} \, d w
& \le
\int_{\Omega} \mathcal{A}(x, \nabla u_{k}) \cdot \nabla u_{k} \, dx
=
\int_{\Omega} u_{k} \left( u_{k} + \frac{1}{k} \right)^{q - 1}  \, d \sigma_{k}
\\
& \le
\int_{\Omega} u_{k}^{q} \, d \sigma
\le
C_{1}^{q} \left( \int_{\Omega} |\nabla u_{k}|^{p} \, d w \right)^{ \frac{q}{p} }.
\end{split}
\]
Therefore, $u(x) = \lim_{k \to \infty} u_{k}(x)$ belongs to $H_{0}^{1, p}(\Omega; w)$, and
\[
\| \nabla u \|_{L^{p}(\Omega; w)}
\le
\liminf_{k \to \infty}
\| \nabla u_{k} \|_{L^{p}(\Omega; w)}
\le
\alpha^{\frac{-1}{p - q}} C_{1}^{ \frac{q}{p - q} }.
\]
Meanwhile, by the argument in the proof of Theorem \ref{thm:singular},
$u$ satisfies \eqref{eqn:singular} in the sense of weak solutions.
The uniqueness follows from Remark \ref{rem:uniqueness}.

Conversely, assume the existence of $u$. By \eqref{eqn:embedding_sing} and H\"{o}lder's inequality,
\[
\begin{split}
\int_{\Omega} \varphi^{q} \, d \sigma
& =
\int_{\Omega} u^{q (1 - q)} \left( \varphi u^{q - 1} \right)^{q} \, d \sigma 
\\
& \le
\left( \int_{\Omega} u^{q} \, d \sigma \right)^{1 - q} 
\left( \int_{\Omega} \varphi u^{q - 1} \, d \sigma \right)^{q}
\\
& \le
\beta \left( \int_{\Omega} |\nabla u|^{p} \, d w \right)^{\frac{p - q}{p}}
\left( \int_{\Omega} |\nabla \varphi|^{p} \, d w \right)^{\frac{q}{p}},
\quad \forall \varphi \in C_{c}^{\infty}(\Omega), \ \varphi \ge 0.
\end{split}
\]
Therefore, \eqref{eqn:trace_ineq} holds with $C_{1} \le \beta^{\frac{1}{q}} \| \nabla u \|_{L^{p}(\Omega; w)}^{\frac{p - q}{q}}$.
\end{proof}

\begin{proof}[Proof of Theorem \ref{thm:bdd_sol}]
Assume the existence of $v$.
By Theorem \ref{thm:singular} and \cite[Theorem 7.25]{MR2305115}, there exists
a bounded weak solution $u \in H^{1, p}_{\loc}(\Omega; w) \cap L^{\infty}(\Omega)$ to \eqref{eqn:singular}.
By the weak Harnack inequality, $h(u)$ is locally bounded in $\Omega$.
Fix $x_{0} \in \Omega$.
By \cite[Theorem 4.20]{MR1264000} (see also \cite[Corollary 3.17]{MR1386213} for weighted equations), 
$u$ is continuous at $x_{0}$ if and only if
\[
\lim_{R \to 0} \sup_{x \in B(x_{0}, R)} \W_{p, w}^{R} (h(u) \sigma)(x) = 0.
\]
On the other hand, since $v$ is continuous at $x_{0}$, by the same reason,
\[
\lim_{R \to 0} \sup_{x \in B(x_{0}, R)} \W_{p, w}^{R} \sigma(x) = 0.
\]
Consequently, $u$ is continuous in $\Omega$.
The boundary continuity of $u$ follows from \eqref{eqn:bound_sing}.

Let $u, v \in H^{1, p}_{\loc}(\Omega; w) \cap C(\cl{\Omega})$ be continuous weak solutions to \eqref{eqn:singular}.
Assume that there exists $x \in \Omega$ such that $u(x) > v(x)$.
Then the open set $D = \{ x \in \Omega \colon u(x) > v(x) + \epsilon \}$ is not empty for $\epsilon > 0$ small.
Recall that $u$ vanishes on $\partial \Omega$ continuously and that $v \ge 0$ in $D$.
Thus $\cl{D} \Subset \Omega$ and $u, v \in H^{1, p}(D; w)$.
Using the test function $(u - v - \epsilon) \in H_{0}^{1, p}(D; w)$, we get
\[
\begin{split}
&
\int_{D} \left( \A(x, \nabla u) - \A(x, \nabla v) \right)  \cdot \nabla (u - v) \, dx
\\
& =
\int_{D} \left( \A(x, \nabla u) - \A(x, \nabla v) \right)  \cdot \nabla (u - v - \epsilon) \, dx
\\
& =
\int_{D} (u - v - \epsilon) \left( h(u) -h(v) \right) \, d \sigma \le 0.
\end{split}
\]
Thus, $\nabla u = \nabla v$ a.e. in $D$ and $(u - v - \epsilon) = 0$ in $D$. This contradicts the assumption.

Conversely, assume the existence of $u$. 
By Theorem \ref{thm:comparison_principle},
\[
\begin{split}
u(x)
& \ge
\w_{\A}\left( h(u) \sigma \right)(x) 
\\
& \ge
\w_{\A}\left( h(\sup_{\Omega} u) \sigma \right)(x)
=
h(\sup_{\Omega} u)^{\frac{1}{p - 1}} \w_{\A} \sigma(x), \quad \forall x \in \Omega.
\end{split}
\]
Therefore, $v = \w_{\A} \sigma$ is a bounded weak solution to \eqref{eqn:p-laplace}, and
\begin{equation}\label{eqn:bound_for_sing_eq}
\begin{split}
v(x) \le \frac{u(x)}{ h(\sup_{\Omega} u)^{\frac{1}{p - 1}} },
\quad \forall x \in \Omega.
\end{split}
\end{equation}
The interior regularity of $v$ follows from the same argument as above.
\end{proof}

\section*{Acknowledgments}
The author would like to thank Professor Verbitsky for providing useful information on the contents of  \cite{MR3881877}.
This work was supported by JSPS KAKENHI Grant Number JP18J00965 and JP17H01092.
%I would also like to thank my family for their support.

%%%%%%%%%%%%%%%%%%%%%%%%%%%%%%%%%%%%%%%%
% References
%%%%%%%%%%%%%%%%%%%%%%%%%%%%%%%%%%%%%%%%

%\bibliographystyle{amsalpha} % plain, alpha, abbrv, unsrt
\bibliographystyle{abbrv}
\bibliography{reference}
% \nocite{*}

%%%%%%%%%%%%%%%%%%%%%%%%%%%%%%%%%%%%%%%%
% The end of the document
%%%%%%%%%%%%%%%%%%%%%%%%%%%%%%%%%%%%%%%%

\end{document}